\documentclass{tran-l}
\usepackage[mathscr]{eucal}
\usepackage[breaklinks=true]{hyperref}

\usepackage{xcolor}
\usepackage{amsmath,amsthm,amsfonts,amssymb,tocvsec2}
\usepackage{mathdots}
\usepackage{mathtools}
\usepackage{caption}
\usepackage{subcaption}
\usepackage{float}
\usepackage{multicol}
\usepackage{enumitem}

\theoremstyle{plain}
\newtheorem{theorem}{Theorem}[section]

\newtheorem{lemma}[theorem]{Lemma}
\newtheorem{prop}[theorem]{Proposition}
\newtheorem{cor}[theorem]{Corollary}
\newtheorem{utheorem}{\textrm{\textbf{Theorem}}}

\theoremstyle{definition}
\newtheorem{defn}[theorem]{Definition}

\newtheorem{rem}[theorem]{Remark}

\numberwithin{equation}{section}

\DeclareMathOperator{\sgn}{sgn}

\DeclareMathOperator{\diag}{diag}

\begin{document}
	\title[Entrywise preservers of sign regularity]{Entrywise preservers of sign regularity}
	
		\author{Projesh Nath Choudhury and Shivangi Yadav}
		\address[P.N.~Choudhury]{Department of Mathematics, Indian Institute of Technology Gandhinagar, Gujarat 382355, India}
		\email{\tt projeshnc@iitgn.ac.in}
		\address[S.~Yadav]{Department of Mathematics, Indian Institute of Technology Gandhinagar, Gujarat 382355, India}
		\email{\tt shivangi.yadav@iitgn.ac.in, shivangi97.y@gmail.com}
	
	
	\begin{abstract}
		Entrywise functions preserving positivity and related notions have a rich history, beginning with the seminal works of Schur, P\'olya--Szeg\H{o}, Schoenberg, and Rudin. Following their classical results, it is well-known that entrywise functions preserving positive semidefiniteness for matrices of all dimensions must be real analytic with non-negative Taylor coefficients. These works were taken forward in the last decade by Belton, Guillot, Khare, Putinar, and Rajaratnam. Recently, Belton--Guillot--Khare--Putinar [\textit{J. d'Analyse Math.} 2023] classified all functions that entrywise preserve totally positive (TP) and totally non-negative (TN) matrices. In this paper, we study entrywise preservers of strictly sign regular and sign regular matrices -- a class that includes TP/TN matrices as special cases and was first studied by Schoenberg in 1930 to characterize variation diminution. Our main results provide complete characterizations of entrywise transforms of rectangular matrices which preserve: (i)~sign regularity and strict sign regularity, as well as (ii)~sign regularity and strict sign regularity with a given sign pattern.
	\end{abstract}
	
	\subjclass[2020]{15B48 (primary); 39B62, 15A15 (secondary)}
	
	\keywords{Total positivity, strict sign regularity, sign regularity, entrywise preserver, Cauchy functional equation, Cauchy--Binet formula}
	
	\maketitle
	
	\vspace*{-11mm}
	\settocdepth{section}
	\tableofcontents
	
	\section{Introduction and main results}
	In this work, we continue recent progresses in the evergreen field of preserver problems, by classifying the entrywise preservers of strictly sign regular and sign regular matrices. Recall that a real rectangular matrix $A$ is said to be \textit{strictly sign regular} (SSR) if there exists a sequence of signs $\epsilon_k\in\{\pm1\}$ such that every $k\times k$ minor of $A$ has sign $\epsilon_k$ for all $k\geq1$. If the signs $\epsilon_k$ are all positive, the matrix $A$ is called \textit{totally positive}. When zero minors are also allowed in the preceding definitions, the matrix $A$ is called \textit{sign regular} (SR) (respectively, \textit{totally non-negative}). These matrices have featured in diverse areas in mathematics, including analysis, approximation theory, cluster algebras, combinatorics, differential equations, Gabor analysis, integrable systems, matrix theory, probability and statistics, Lie theory, and representation theory \cite{Ando87,BFZ96,Bre95,C22,fallat-john,FZ00,FZ02,gantmacher-krein,GRS18,Karlin64,K68,Karlinsplines,Khare20,KW14,Lo55,Lu94,pinkus,Ri03,Schoenberg46,S55,Whitney}.
	
	One of the earliest known results on SR matrices can be traced back to Schoenberg~\cite{S30}, who in 1930 studied the variation diminishing property using sign regularity. Variation diminution is a prominent property in total positivity theory -- it was studied by Gantmacher--Krein, Motzkin, and Brown--Johnstone--MacGibbon to characterize SSR and SR matrices with various constraints (see~\cite{BJM81,GK50,Mot36}). In our
	recent work~\cite{CY-VDP23}, we got rid of all these constraints and provided a complete characterization of when an arbitrary real rectangular matrix is SR and SSR using variation diminution. We further strengthen these characterizations by providing a single test vector with alternating sign coordinates for each square submatrix.
	
	A fundamental problem in the literature has been that of identifying the preservers of various structures, including forms of positivity and related notions. These problems have been studied in different frameworks, such as the statistics of big data \cite{BL08,HR11,RLZ09} and global optimization algorithms based on positive definite polynomials \cite{BPR12,HMPV09}. The first result in the preservers literature goes back to Frobenius~\cite{Frobenius1897}, who characterized the general form of determinant preserving linear maps in 1897. In prior recent work~\cite{CY-LPP24}, we classified all surjective linear operators that preserve sign regularity and strict sign regularity -- with fixed as well as arbitrary sign patterns.
	
	The goal of this work is to characterize entrywise functions that preserve the set of SR/SSR matrices of a given dimension. The study of entrywise calculus was initiated by Schur~\cite{Schur11} in 1911, when he proved the celebrated Schur product theorem. Building on this, P\'olya--Szeg\H{o}~\cite{PS25} in 1925 showed that real analytic functions with non-negative Taylor coefficients when applied entrywise to positive semidefinite matrices preserve positive semidefiniteness. The converse to this result -- that these are the only continuous functions which preserve positive semidefiniteness was shown by Schoenberg~\cite{Schoenberg42} in 1942. In 1959, Rudin~\cite{Rudin59} strengthened Schoenberg's result by removing the continuity hypothesis. Since then, this has been extensively studied in various settings~\cite{BGKP16,GKR16,GKR_16,GKR17}.
	
	We now come to the question of characterizing entrywise preservers of sign regularity and strict sign regularity. Recently, Belton--Guillot--Khare--Putinar~\cite{BGKP20} classified all entrywise preservers for the classes of totally positive and totally non-negative matrices. In this paper, we extend these results in two ways. First, we classify the entrywise preservers of SR and SSR matrices for each specified sign pattern. Second, we characterize the entrywise preservers of SR and SSR matrices allowing all sign patterns. To state our results, we introduce the following definitions and notations.
	
	\begin{defn}
		Let $m,n\geq1$ denote integers throughout this manuscript.
		\begin{itemize}
			\item[(i)] A square matrix whose all anti-diagonal elements are $1$ and the remaining entries are $0$ is called an \textit{exchange matrix}. Let $P_n$ denote the $n\times n$ exchange matrix:
			\[P_n:=\begin{pmatrix}
				0 & \cdots & 1 \\
				\vdots & \iddots & \vdots \\
				1 & \cdots & 0 \\
			\end{pmatrix} \in \{ 0, 1 \}^{n\times n}.\]
			
			\item[(ii)] The \textit{sign pattern} of an $m\times n$ SR/SSR matrix $A$ is an ordered tuple $\epsilon=(\epsilon_1,\ldots,\epsilon_{\min\{m,n\}})$, where each $\epsilon_k\in\{\pm1\}$ represents the sign of the $k\times k$ minors of $A$.
			
			\item[(iii)] Let $\epsilon=(\epsilon_1,\ldots,\epsilon_{\min\{m,n\}})$ be a given sign pattern. We use the notation SSR$(\epsilon)$ to denote $m\times n$ SSR matrices whose $k\times k$ minors are all non-zero and have sign $\epsilon_k$ for all $1\leq k\leq \min\{m,n\}$. Similarly, SR$(\epsilon)$ denotes $m\times n$ SR matrices whose $k\times k$ minors are either zero or have sign $\epsilon_k$ for all $1\leq k\leq \min\{m,n\}$. Note that $\epsilon_1$ determines if the entries are all positive/negative or non-negative/non-positive.
			
			\item[(iv)] Let $f:I\subseteq\mathbb{R}\to\mathbb{R}$ be a function. For an $m\times n$ matrix $A=(a_{ij})$ with entries $a_{ij}\in I$, the \textit{entrywise action} of $f$ on $A$ is defined by
			\[f[A] := (f(a_{ij}))\in\mathbb{R}^{m\times n}.\]
			In particular, for the power function $f(x)=x^{\alpha}$, with $\alpha\in\mathbb{R}$, we denote the entrywise application of $f$ to $A$ by $A^{\circ\alpha}:=(a_{ij}^{\alpha})$. We set $0^0:=1$.
				
			\item[(v)] Let $I\subseteq\mathbb{R}$ and $f:I\to\mathbb{R}$ be a function. Given a set $S$ of matrices with entries in $I$, we say that $f$ \textit{entrywise preserves} $S$ if, for every $A\in S$, the matrix $f[A]\in S$.
				
		\end{itemize}
	\end{defn}
	Our first main result characterizes the entrywise functions that preserve the class of $m\times n$ SR$(\epsilon)$ matrices for any given sign pattern $\epsilon$ with $\epsilon_1=1$.
	
	\begin{utheorem}\label{Theorem_SR_Fixed_Classification}
		Let $m,n\geq1$ be integers, and let $\epsilon=(\epsilon_1,\ldots,\epsilon_d)$ be a given sign pattern with $\epsilon_1=1$, where $d:=\min\{m,n\}$. Let $f:[0,\infty)\to\mathbb{R}$ be a function. Then the following statements are equivalent.
		\begin{itemize}
			\item[(1)] $f[-]$ preserves the class of $m\times n$ SR$(\epsilon)$ matrices.
			\item[(2)] $f[-]$ preserves the class of $d\times d$ SR$(\epsilon)$ matrices.
			\item[(3)]
			\begin{itemize}
				\item[(a)] For $d=1:$ The function $f$ is non-negative.
				\item[(b)] For $d=2:$ $f(x)=c\sgn(x)$ or $f(x)=cx^{\alpha}$ for some $\alpha\in[0,\infty)$ and some $c\geq0$.
				\item[(c)] For $d=3:$ $f(x) = \begin{cases}
					c\sgn(x) ~ \text{or} ~ cx^{\alpha} ~ \text{for some} ~ \alpha\in[0,1] ~ \text{and some} ~ c\geq0, &\text{if} ~ \epsilon_2\neq\epsilon_3, \\
					cx^{\alpha} ~ \text{for some} ~ \alpha\in\{0\}\cup[1,\infty) ~ \text{and some} ~ c\geq0, &\text{if} ~ \epsilon_2=\epsilon_3.
				\end{cases}$
				
				\item[(d)] For $d\geq4:$ $f(x) = \begin{cases}
					c\sgn(x) ~ \text{or} ~ cx^{\alpha} ~ \text{for some} ~ \alpha\in\{0,1\} ~ \text{and some} ~ c\geq0, &\text{if} ~ \epsilon_2\neq\epsilon_3, \\
					cx^{\alpha} ~ \text{for some} ~ \alpha\in\{0,1\} ~ \text{and some} ~ c\geq0, &\text{if} ~ \epsilon_2=\epsilon_3.
				\end{cases}$
			\end{itemize}
		\end{itemize}
	\end{utheorem}
	
	This theorem ensures that any entrywise preserver of $m\times n$ SR$(\epsilon)$ matrices with $\epsilon_1=1$ must belong to the specific class of functions identified above. We now present the corresponding result for SSR$(\epsilon)$ matrices.
	
	\begin{utheorem}\label{Theorem_SSR_Fixed_Classification}
		Let $m,n\geq1$ be integers, and let $\epsilon=(\epsilon_1,\ldots,\epsilon_d)$ be a given sign pattern with $\epsilon_1=1$, where $d:=\min\{m,n\}$. Let $f:(0,\infty)\to\mathbb{R}$ be a function. Then the following statements are equivalent.
		\begin{itemize}
			\item[(1)] $f[-]$ preserves the class of $m\times n$ SSR$(\epsilon)$ matrices.
			\item[(2)] $f[-]$ preserves the class of $d\times d$ SSR$(\epsilon)$ matrices.
			\item[(3)]
			\begin{itemize}
				\item[(a)] For $d=1:$ The function $f$ is positive.
				\item[(b)] For $d=2:$ $f(x)=cx^{\alpha}$ for some $\alpha\in(0,\infty)$ and some $c>0$.
				\item[(c)] For $d=3:$ $f(x) = \begin{cases}
					cx^{\alpha} ~ \text{for some} ~ \alpha\in(0,1] ~ \text{and some} ~ c>0, &\text{if} ~ \epsilon_2\neq\epsilon_3, \\
					cx^{\alpha} ~ \text{for some} ~ \alpha\in[1,\infty) ~ \text{and some} ~ c>0, &\text{if} ~ \epsilon_2=\epsilon_3.
				\end{cases}$
				\item[(d)] For $d\geq4:$ $f(x)=cx$ for some $c>0$.
			\end{itemize}
		\end{itemize}
	\end{utheorem}
	
	For $\min\{m,n\}\geq4$, the above result shows that the entrywise preservers of $m\times n$ SSR$(\epsilon)$ matrices are linear -- specifically, positive homotheties.
	
	\begin{rem}
		As an application of Theorems~\ref{Theorem_SR_Fixed_Classification}~and~\ref{Theorem_SSR_Fixed_Classification}, we classify entrywise preservers of SR$(\epsilon)$ and SSR$(\epsilon)$ matrices with $\epsilon_1=-1$. Thus, we have a complete characterization of entrywise preservers of SR$(\epsilon)$/SSR$(\epsilon)$ matrices for an arbitrary sign pattern $\epsilon$ (see Corollaries~\ref{Corollary_SR_epsilon_1=-1}~and~\ref{Corollary_SSR_epsilon_1=-1}).
	\end{rem}
	
	Now, rather than considering the set of SR/SSR matrices with a specified sign pattern, one can consider a broader class -- the set of all $m\times n$ SR/SSR matrices for any given $m,n\geq1$. In this setting, the transformed matrix is required to remain SR or SSR, although its sign pattern may differ from that of the original matrix. The following theorem provides a complete characterization of entrywise preservers of all $m\times n$ SR matrices for each fixed $m,n\geq1$.
	
	\begin{utheorem}\label{Theorem_SR_Non-Fixed_Classification}
		Let $m,n\geq1$ be integers and define $d:=\min\{m,n\}$. Suppose $f:\mathbb{R}\to\mathbb{R}$ is a function. Then the following statements are equivalent.
		\begin{itemize}
			\item[(1)] $f[-]$ preserves the set of all $m\times n$ SR matrices.
			\item[(2)] We have the following two cases:
			\begin{itemize}
				\item [(a)] If $m=n$, then we have:
				\begin{itemize}
					\item[(i)] For $n=1:$ $f$ is any real-valued function.
					
					\item[(ii)] For $n=2:$ If $f(0)=0$, then $f|_{(-\infty,0)}$ and $f|_{(0,\infty)}$ are each either non-positive or non-negative functions. Otherwise, $f$ is either a non-positive or a non-negative function on $\mathbb{R}$.
					
					\item[(iii)] For $n=3:$ If $f(0)\neq0$, then $f\equiv f(0)$ on $\mathbb{R}$. If $f(0)=0$, then
					\begin{align*}
						f(x) =
						\begin{cases}
							c_1 |x|^{\alpha_1} ~ \text{for some} ~ \alpha_1\geq0 ~ \text{and some} ~ c_1\in\mathbb{R}, & \text{if } x<0, \\
							c_2 x^{\alpha_2} ~ \text{for some} ~ \alpha_2\geq0 ~ \text{and some} ~ c_2\in\mathbb{R}, & \text{if } x>0.
						\end{cases}
					\end{align*}
					
					\item[(iv)] For $n\geq4:$ Either $f$ is a non-zero constant function on $\mathbb{R}$, or
					\begin{align*}
						f(x)=\begin{cases}
							c_1|x| ~ \text{for some} ~ c_1\in\mathbb{R}, &\text{if} ~ x\leq0,\\
							c_2x ~ \text{for some} ~ c_2\in\mathbb{R}, &\text{if} ~ x\geq0.
						\end{cases}
					\end{align*}
				\end{itemize}
				
				\item [(b)] If $m\neq n$, then we have:
				\begin{itemize}
					\item[(i)] For $d=1:$ same form as in case (a)(ii).
					
					\item[(ii)] For $d=2:$ same form as in case (a)(iii).
					
					\item[(iii)] For $d\geq3:$ same form as in case (a)(iv).
				\end{itemize}
			\end{itemize}	
		\end{itemize}
	\end{utheorem}
	
	Unlike the case with a fixed sign pattern, the entrywise preservers of the set of all SR matrices of size $m\times n$ need not coincide with those for $d\times d$ matrices, where $d:=\min\{m,n\}$. Our next theorem provides the similar result for SSR matrices.
	
	\begin{utheorem}\label{Theorem_SSR_Non-Fixed_Classification}
		For integers $m,n\geq1$ with $(m,n)\neq(2,2)$, define $d:=\min\{m,n\}$. Let $f:\mathbb{R}\setminus\{0\}\to\mathbb{R}$ be a function. Then the following statements are equivalent.
		\begin{itemize}
			\item[(1)] $f[-]$ preserves the set of all $m\times n$ SSR matrices.
			\item[(2)] We have the following two cases:
			\begin{itemize}
				\item[(a)] If $m=n$, then we have:
				\begin{itemize}
					\item[(i)] For $n=1:$ $f(x)\neq0$ for all $x\in\mathbb{R}\setminus\{0\}$.
					
					\item[(ii)] For $n\geq3:$ $f(x)=\begin{cases}
						c_1|x| ~ \text{for some} ~ c_1\in\mathbb{R}\setminus\{0\}, &\text{if} ~ x<0,\\
						c_2x ~ \text{for some} ~ c_2\in\mathbb{R}\setminus\{0\}, &\text{if} ~ x>0.
					\end{cases}$
				\end{itemize}
				
				\item[(b)] If $m\neq n$, then we have:
				\begin{itemize}
					\item[(i)] For $d=1:$ $f|_{(-\infty,0)}$ and $f|_{(0,\infty)}$ are each either a positive function or a negative function.
					
					\item[(ii)] For $d=2:$ $f(x)=\begin{cases}
						c_1|x|^{\alpha_1} ~ \text{for some} ~ \alpha_1, c_1\in\mathbb{R}\setminus\{0\}, &\text{if} ~ x<0,\\
						c_2x^{\alpha_2} ~ \text{for some} ~ \alpha_2, c_2\in\mathbb{R}\setminus\{0\}, &\text{if} ~ x>0.
					\end{cases}$
					
					\item[(iii)] For $d\geq3:$ same form as in case (a)(ii).
				\end{itemize}
			\end{itemize}
		\end{itemize}
	\end{utheorem}
	
	\begin{rem}
		Note that entrywise preservers of all $2\times2$ SSR matrices are not covered by the above theorem. This case is addressed separately in Lemma~\ref{Lemma_2by2_SSR}.
	\end{rem}
	
	We conclude this discussion of our main results with two remarks. First, SSR matrices are harder to analyze than TP matrices, which have been the focus of tremendous study in the literature. This paper is an attempt to better understand the more challenging picture of preservers of SSR matrices. Second, it is striking that the solutions to this more challenging preserver problem have a structured and precise form as in the above results.
	
	\subsection*{Organization of the paper}
	
	The remainder of the paper is dedicated to proving our main results. We begin in Section~\ref{Section_Fixed_Sign_pattern} by recording some preliminary observations. We then establish several foundational results, which are later combined to prove Theorems~\ref{Theorem_SR_Fixed_Classification}~and~\ref{Theorem_SSR_Fixed_Classification}. These theorems, together with Corollaries~\ref{Corollary_SR_epsilon_1=-1}~and~\ref{Corollary_SSR_epsilon_1=-1}, provide a complete characterization of functions that entrywise preserve SR$(\epsilon)$ and SSR$(\epsilon)$ matrices for any given sign pattern $\epsilon$ in each fixed dimension. In the final section, we prove Theorems~\ref{Theorem_SR_Non-Fixed_Classification}~and~\ref{Theorem_SSR_Non-Fixed_Classification}, which classify entrywise preservers of (strict) sign regularity allowing all sign patterns.
	
	\section{Entrywise preservers of sign regularity for a fixed sign pattern}\label{Section_Fixed_Sign_pattern}
	
	The aim of this section is to classify all real-valued functions $f$ such that $f[-]$ preserves the set of $m\times n$ SR$(\epsilon)$/SSR$(\epsilon)$ matrices, for any given integers $m,n\geq1$ and a sign pattern $\epsilon$. We begin by considering the case where all entries are positive. In this setting, we first identify the entrywise power functions that preserve (strict) sign regularity for a given sign pattern, and then characterize all entrywise transforms preserving them. Finally, using these results, we classify entrywise preservers of SR$(\epsilon)$/SSR$(\epsilon)$ matrices with $\epsilon_1=-1$. We first record some observations that follow from the Cauchy--Binet formula.
	
	\begin{rem}\label{Remark_SSR_3}
		For any $\alpha\in\mathbb{R}$ and an $n\times n$ diagonal matrix $D=\diag(d_1,\ldots,d_n)$ with $d_i>0$ for all $i$, define $\widehat{D}^{\circ\alpha}:= \diag(d_1^{\alpha},\ldots,d_n^{\alpha})$.
		\begin{itemize}
			\item[(i)] \textit{Let $A$ be an $m\times n$ matrix with positive entries, and let $E=\diag(e_1,\ldots,e_m)\in\mathbb{R}^{m\times m}$ and $F=\diag(f_1,\ldots,f_n)\in\mathbb{R}^{n\times n}$ be diagonal matrices with all $e_i,f_j>0$. Then $A$ is SSR$(\epsilon)$ if and only if $EAF$ is SSR$(\epsilon)$.} This implies that for any $\alpha\in\mathbb{R}$, the matrix $A^{\circ\alpha}$ is SSR$(\epsilon)$ if and only if $\widehat{E}^{\circ\alpha}A^{\circ\alpha}\widehat{F}^{\circ\alpha}$ is SSR$(\epsilon)$. Since $(EAF)^{\circ\alpha} = \widehat{E}^{\circ\alpha}A^{\circ\alpha}\widehat{F}^{\circ\alpha}$, it suffices to determine the values of $\alpha$ for which $(EAF)^{\circ\alpha}$ is SSR$(\epsilon)$ in order to conclude the same for $A^{\circ\alpha}$.
			
			\item[(ii)] \textit{Let $A\in\mathbb{R}^{m\times n}$ and set $d:=\min\{m,n\}$. Then $A$ is SSR$(\epsilon)$ if and only if $AP_n$ is SSR$(\epsilon^{\prime})$, where $\epsilon_i^{\prime} = (-1)^{\left\lfloor \frac{i}{2}\right\rfloor}\epsilon_i$ for $i=1,\ldots,d$}. This follows from the identity $\det P_n = (-1)^{\left\lfloor \frac{n}{2}\right\rfloor}$. 
			
			\item[(iii)] Let $f$ be a real-valued function, and let $A$ be an $m\times n$ SSR$(\epsilon)$ matrix whose entries lie in the domain of $f$. Note that $f[A]P_n=f[AP_n]$. By the preceding point, the matrix $f[A]$ is SSR$(\epsilon)$ if and only if $f[AP_n]$ is SSR$(\epsilon^{\prime})$, where $\epsilon_i^{\prime}=(-1)^{\left\lfloor \frac{i}{2}\right\rfloor}\epsilon_i$ for $i=1,\ldots,\min\{m,n\}$. This implies that the entrywise preservers corresponding to the sign patterns $\epsilon$ and $\epsilon^{\prime}$ are identical. Consequently, this result reduces the number of sign patterns that need to be examined when characterizing entrywise preservers of SSR$(\epsilon)$ matrices.
		\end{itemize}
		These observations remain valid if we replace SSR$(\epsilon)$ with SR$(\epsilon)$, for all $\alpha\geq0$.
	\end{rem}
	
	We next recall Descartes' rule of signs for exponential polynomials, which gives a bound on the number of their real zeros in terms of the sign changes in their coefficients.
		
	\begin{theorem}\cite[Descartes' rule of signs]{Jameson}\label{Jameson}
		Let $c_j\in\mathbb{R}$, not all zero, and let $\alpha_1>\cdots>\alpha_n$ be real numbers. Define $F:\mathbb{R}\to\mathbb{R}$ by
		\begin{align*}
			F(t) = \displaystyle\sum_{j=1}^{n}c_je^{\alpha_jt}.
		\end{align*}
		Then the number of real zeros of $F$, counted with their multiplicities, is at most the number of sign changes in the sequence $(c_1,\ldots,c_n)$ after removing all zero entries.
	\end{theorem}
	
	\subsection{Entrywise powers preserving SSR$(\epsilon)$ and SR$(\epsilon)$ matrices}
	
	Using Remark~\ref{Remark_SSR_3} and Theorem~\ref{Jameson}, we now determine which entrywise power functions preserve $3\times3$ SSR$(\epsilon)$ matrices.
	
	\begin{theorem} \label{Theorem_SSR_3_powers}
		Let $A$ be a $3\times3$ SSR$(\epsilon)$ matrix with sign pattern $\epsilon=(\epsilon_1,\epsilon_2,\epsilon_3)$ and $\epsilon_1=1$. Then $A^{\circ\alpha}$ remains SSR$(\epsilon)$ for all $\alpha\in(0,1]$ if $\epsilon_2\neq\epsilon_3$, and for all $\alpha\in[1,\infty)$ if $\epsilon_2=\epsilon_3$.
	\end{theorem}
	\begin{proof}
		Let $A\in\mathbb{R}^{3\times3}$ be an SSR$(\epsilon)$ matrix with $\epsilon_1=1$, and define $f(x):=x^{\alpha}$ for all $x>0$, with $\alpha\in\mathbb{R}$. Note that applying $f$ entrywise preserves the sign of all $1\times1$ minors of $A$ for every $\alpha\in\mathbb{R}$, and preserves the sign of all $2\times2$ minors if and only if $\alpha>0$. To complete the proof, it suffices to classify $\alpha\in(0,\infty)$ such that $f[-]$ preserves the sign of the determinant of $A$.
		
		Let $A:=(a_{ij})$, and define the diagonal matrices $E$ and $F$ as follows:
		\begin{align*}
			E:=\begin{pmatrix}
				\frac{1}{a_{11}} & 0 & 0 \\
				0 & \frac{1}{a_{21}} & 0 \\
				0 & 0 & \frac{1}{a_{31}}
			\end{pmatrix}, ~ F:=\begin{pmatrix}
				1 & 0 & 0 \\
				0 & \frac{a_{11}}{a_{12}} & 0 \\
				0 & 0 & \frac{a_{11}}{a_{13}}
			\end{pmatrix}. ~ \text{Then} ~ EAF = \begin{pmatrix}
				1 & 1 & 1 \\
				1 & \frac{a_{11}a_{22}}{a_{12}a_{21}} & \frac{a_{11}a_{23}}{a_{21}a_{13}} \\[5pt]
				1 & \frac{a_{11}a_{32}}{a_{12}a_{31}} & \frac{a_{11}a_{33}}{a_{13}a_{31}}
			\end{pmatrix}. 
		\end{align*}
		By Remark~\ref{Remark_SSR_3}(i), without loss of generality we may assume
		\begin{equation*}
			A:=\begin{pmatrix}
				1 & 1 & 1 \\
				1 & x_1 & x_2 \\
				1 & x_3 & x_4
			\end{pmatrix}, ~ \text{where} ~ x_i>0 ~\text{for all} ~ i.
		\end{equation*}
		The matrix $A$ is SSR$(\epsilon)$ if and only if the following conditions hold: (a)~$\epsilon_2(x_i-1)>0$ for all $i$, (b)~$\epsilon_2(x_1x_4-x_2x_3)>0$, (c)~$\epsilon_2(x_j-x_i)>0$ for all $i<j$, except when $(i,j)=(2,3)$, and (d)~$\epsilon_3\det A>0$. For $\alpha\in\mathbb{R}$, define the function
		\begin{align}
			F(\alpha):=\det A^{\circ\alpha} & =(x_1x_4)^\alpha-(x_2x_3)^\alpha-x_4^\alpha+x_2^\alpha+x_3^\alpha-x_1^\alpha  \label{DirichletPolynomial}\\
			& = e^{(\log x_1x_4)\alpha}-e^{(\log x_2x_3)\alpha}-e^{(\log x_4)\alpha}+e^{(\log x_2)\alpha}+e^{(\log x_3)\alpha}-e^{(\log x_1)\alpha}. \label{DirichletPolynomial_exp}
		\end{align}
		
		Note that $F$ is continuous on $\mathbb{R}$. For $\epsilon_2=1$, we have the following properties of $F$:
		\begin{itemize}
			\item[(i)] From \eqref{DirichletPolynomial}, we have $\displaystyle{\lim_{\alpha\to\infty}}F(\alpha)=\infty$ as $x_2,x_3,x_4>x_1>1$ and $x_1x_4>x_2x_3.$
			
			\item[(ii)] $F(1)=\det A$, so $\text{sign}\left(F(1)\right)=\epsilon_3$.
			
			\item[(iii)] The number of sign changes in the coefficients of $F$ in \eqref{DirichletPolynomial_exp} remains invariant under the transposition of the terms $-e^{(\log x_2x_3)\alpha}$ and $-e^{(\log x_4)\alpha}$, as well as $e^{(\log x_2)\alpha}$ and $e^{(\log x_3)\alpha}$. Hence, by Theorem~\ref{Jameson}, $F$ has at most three real roots.
			
			\item[(iv)] One can carry out a Taylor expansion of $F$ at $\alpha=0$ and show that
			\[F(\alpha)= \left(\log x_1\log x_4-\log x_2\log x_3\right)\alpha^2 + O(\alpha^3).\]
			It follows that $F$ has a root of multiplicity at least two at $\alpha=0$.
		\end{itemize}
		Using the above properties of $F$ valid for $\epsilon_2=1$ and $\epsilon_3\in\{\pm1\}$, we now complete the proof by considering two cases.
		
		\begin{itemize}[leftmargin=0pt,label={}]
			\item \textbf{Case 1.} $\epsilon_2\neq\epsilon_3$: We claim that $\epsilon_3\det A^{\circ\alpha}>0$ for all $0<\alpha\leq1$. We begin by considering the case $\epsilon_1=\epsilon_2$, i.e., $\epsilon_1=\epsilon_2=1$ and $\epsilon_3=-1$. Assume for contradiction that $\epsilon_3F(\alpha_0)<0$ for some $\alpha_0\in(0,1)$. The continuity of $F$ and $\epsilon_3F(1)>0$ imply that $F$ has at least one root in $(\alpha_0,1)$. Furthermore, the behavior $\epsilon_3F(\alpha)\to-\infty$ as $\alpha\to\infty$ implies another root of $F$ in $(1,\infty)$. With two roots at zero already, this gives at least four real roots of $F$, contradicting Theorem~\ref{Jameson}. Thus, $\epsilon_3F(\alpha)\geq0$ for all $\alpha\in(0,1)$. Moreover, if $\epsilon_3F(\alpha_0)=0$, then $\alpha_0$ is a global maximum point of $F$ in $(0,1)$, so $F^{\prime}(\alpha_0)=0$, indicating a double root at $\alpha_0$ and again violating Theorem~\ref{Jameson}.
			
			Next, assume that $A$ is SSR$(\epsilon)$ with $\epsilon_1\neq\epsilon_2$. Then the matrix $AP_3$ becomes SSR$(\epsilon^{\prime})$, with $\epsilon_1^{\prime}=\epsilon_2^{\prime}$ and $\epsilon_2^{\prime}\neq\epsilon_3^{\prime}$. By applying Remark~\ref{Remark_SSR_3}(iii) and using the above case, it follows that $A^{\circ\alpha}$ is SSR$(\epsilon)$ for all $\alpha\in(0,1]$. Thus, $\epsilon_3\det A^{\circ\alpha}>0$ for all $0<\alpha\leq1$.
			
			\item \textbf{Case 2.} $\epsilon_2=\epsilon_3$: We show that $\epsilon_3\det A^{\circ\alpha}>0$ for all $\alpha\geq1$. We begin with the case $\epsilon_1=\epsilon_2$. One can adapt the argument used in the first half of the above case, i.e., assume the contrary and then use the properties of $F$ to deduce that $F$ has (at least) four real roots, which is false. 
			
			Now consider the case $\epsilon_1\neq\epsilon_2$. As before, we consider the matrix $AP_3$ to reduce to the case where $\epsilon_1^{\prime}=\epsilon_2^{\prime}=\epsilon_3^{\prime}$, and use Remark~\ref{Remark_SSR_3}(iii) and the above result to conclude that $A^{\circ\alpha}$ remains SSR$(\epsilon)$ for all $\alpha\geq1$. This completes the proof. \qedhere
		\end{itemize}
	\end{proof}
	
	To establish a similar result for $3\times3$ SR$(\epsilon)$ matrices, we require the following classical density result of Gantmacher and Krein.
	
	\begin{theorem}\cite[Gantmacher--Krein]{GK50}\label{Theorem_Gantmacher_Krein}
		Let $m,n\geq1$ be integers and $\epsilon=(\epsilon_1,\ldots,\epsilon_{\min\{m,n\}})$ be a given sign pattern. Then the set of $m\times n$ SSR$(\epsilon)$ matrices is dense in the set of $m\times n$ SR$(\epsilon)$ matrices.
	\end{theorem}
	
	Using Theorems~\ref{Theorem_SSR_3_powers}~and~\ref{Theorem_Gantmacher_Krein}, we now determine the entrywise powers that preserve $3\times3$ SR$(\epsilon)$ matrices.
	
	\begin{cor}\label{Corollary_SR_3_powers}
		Let $A$ be a $3\times3$ SR$(\epsilon)$ matrix with $\epsilon_1=1$. Then $A^{\circ\alpha}$ is SR$(\epsilon)$ for all $\alpha\in[0,1]$ if $\epsilon_2\neq\epsilon_3$, and for all $\alpha\in\{0\}\cup[1,\infty)$ if $\epsilon_2=\epsilon_3$.
	\end{cor}
	\begin{proof}
		Let $A$ be a $3\times3$ SR$(\epsilon)$ matrix with $\epsilon_1=1$. By the above result of Gantmacher--Krein, there exists a sequence of $3\times3$ SSR$(\epsilon)$ matrices $A_k$ such that $\displaystyle{\lim_{k\to\infty}}A_k=A$. Moreover, as $k\to\infty$, $A_k^{\circ\alpha}\to A^{\circ\alpha}$ entrywise for all $\alpha\geq0$. From Theorem~\ref{Theorem_SSR_3_powers}, each $A_k^{\circ\alpha}$ is SSR$(\epsilon)$ for all $\alpha\in(0,1]$ if $\epsilon_2\neq\epsilon_3$, and for all $\alpha\in[1,\infty)$ if $\epsilon_2=\epsilon_3$. By taking the limit as $k\to\infty$, it follows that $A^{\circ\alpha}$ is SR$(\epsilon)$ for all $\alpha\in(0,1]$ if $\epsilon_2\neq\epsilon_3$, and for all $\alpha\in[1,\infty)$ if $\epsilon_2=\epsilon_3$. For $\alpha=0$, $A^{\circ\alpha}$ is the matrix of all ones, which is SR$(\epsilon)$. This concludes the proof.
	\end{proof}
	
	At this point, a natural question arises: Are these the only exponents $\alpha$ for which $A^{\circ\alpha}$ remains SR$(\epsilon)$ whenever $A$ is a $3\times3$ SR$(\epsilon)$ matrix with $\epsilon_1=1$? The following theorem answers this question -- and more generally, provides a characterization for all fixed dimensions $m,n\geq1$.
	
	\begin{rem}
		Note that for the $m\times n$ zero matrix $A$ and any real number $\alpha<0$, the matrix $A^{\circ\alpha}$ is not defined. Therefore, in Theorem~\ref{Theorem_SR_power_classification}, we classify all power functions with non-negative exponents that entrywise preserve $m\times n$ SR$(\epsilon)$ matrices with $\epsilon_1=1$ for each fixed $m,n\geq1$.
	\end{rem}
	
	\begin{theorem}\label{Theorem_SR_power_classification}
		Let $m,n\geq1$ be integers, and define $d:=\min\{m,n\}$. Let $\epsilon=(\epsilon_1,\ldots,\epsilon_d)$ be a given sign pattern, where $\epsilon_1=1$. For $\alpha\in[0,\infty)$, the following statements are equivalent.
		\begin{itemize}
			\item[(1)] The map $x\mapsto x^{\alpha}$ entrywise preserves all $m\times n$ SR$(\epsilon)$ matrices.
			\item[(2)] The map $x\mapsto x^{\alpha}$ entrywise preserves all $d\times d$ SR$(\epsilon)$ matrices.
			\item[(3)] We have either $\alpha=0$, or
			\begin{itemize}
				\item[(a)] For $d=1,2:$ $\alpha\in(0,\infty)$.
				
				\item[(b)] For $d=3:$ $\alpha\in(0,1]$ if $\epsilon_2\neq\epsilon_3$, and $\alpha\in[1,\infty)$ if $\epsilon_2=\epsilon_3$.
				
				\item[(c)] For $d\geq4:$ $\alpha=1$.
			\end{itemize}
		\end{itemize}
	\end{theorem}
	\begin{proof}
		We complete the proof by showing the equivalences (1)$\iff$(2) and (2)$\iff$(3). The implication (2)$\implies$(1) is immediate. To show (1)$\implies$(2), take any $d\times d$ SR$(\epsilon)$ matrix $A_d$, and extend it to an $m\times n$ SR$(\epsilon)$ matrix $A$ by padding with zeros. Now $A^{\circ\alpha}$ is SR$(\epsilon)$, and consequently, all its submatrices including $A_d^{\circ\alpha}$ are SR$(\epsilon)$.
		
		Now, we show (3)$\implies$(2). Applying $x^0$ entrywise to any SR matrix with non-negative entries yields a matrix of all ones, which is sign regular. For $d=3$, the result follows from Corollary~\ref{Corollary_SR_3_powers}, while the remaining cases are immediate. Finally, we prove (2)$\implies$(3). It is easy to verify that all non-negative powers entrywise preserve $1\times1$ and $2\times2$ SR$(\epsilon)$ matrices. We now classify all exponents $\alpha\geq0$ for which the map $x\mapsto x^{\alpha}$ entrywise preserves every $d\times d$ SR$(\epsilon)$ matrix, for each fixed $d\geq3$.
		\begin{itemize}[leftmargin=0pt,label={}]
			\item \textbf{Case 1.} $d=3$: For any $3\times3$ SR$(\epsilon)$ matrix $A$, the below table outlines the known results and the necessary conditions required to show that only the specified powers in (3)(b) entrywise preserve $A$.
			\begin{table}[H] 
				\begin{tabular}{|c|c|c|}
					\hline
					$\boldsymbol{\epsilon}$ &  $\boldsymbol{\alpha: A^{\circ\alpha}}$ \textbf{is SR}$\boldsymbol{(\epsilon)}$ & \textbf{To construct an} \textbf{SR}$\boldsymbol{(\epsilon)}$ \textbf{matrix} $\boldsymbol{A}$ \textbf{such that} \\ [0.5ex] 
					\hline\hline
					
					$\epsilon_2=\epsilon_3$ & $\{0\}\cup[1,\infty)$ & $\epsilon_3\det A^{\circ\alpha}<0$ $\forall$ $\alpha\in(0,1)$ \\ [0.5ex] \hline
					
					$\epsilon_2\neq\epsilon_3$ & $[0,1]$ & $\epsilon_3\det A^{\circ\alpha}<0$ $\forall$ $\alpha\in(1,\infty)$ \\ [0.5ex] \hline
				\end{tabular}
				
				\caption{Table listing the two cases for $3\times3$ SR($\epsilon$) matrices.}
				\label{Table_SR_3}
			\end{table}
			
			We now construct matrices that satisfy the conditions in the last column of Table~\ref{Table_SR_3}. Let
			\begin{equation*}
				A=\begin{pmatrix}
					3 & 1 & 2 \\
					1 & 1 & 4 \\
					1 & 2 & 9 
				\end{pmatrix}.
			\end{equation*}
			This singular matrix is SR$(\epsilon)$, where $\epsilon=(\epsilon_1,\epsilon_2,\epsilon_3)$ with $\epsilon_1=\epsilon_2=1$ and $\epsilon_3\in\{\pm1\}$. Define the function
			\begin{align*}
				F(\alpha):=\det A^{\circ\alpha} = -2^\alpha+2(4^\alpha)-9^\alpha-24^\alpha+27^\alpha.
			\end{align*}
			Note that $F(0)=F(1)=0$ and $F^{\prime}(0)=0$. By Theorem~\ref{Jameson}, $F$ has at most three real roots, so $F$ has no other real roots besides 0 (a double root) and 1. Thus, $F(\alpha)\neq0$ for all $\alpha\in(0,\infty)\setminus\{1\}$. By Corollary~\ref{Corollary_SR_3_powers}, $\epsilon_3F(\alpha)>0$ for all $\alpha>1$, if $\epsilon_2=\epsilon_3$, and $\epsilon_3F(\alpha)>0$ for all $\alpha\in(0,1)$, if $\epsilon_2\neq\epsilon_3$. This shows that 
			\begin{align*}
				\epsilon_3F(\alpha)<\begin{cases}
					0 ~ \text{for all} ~ \alpha\in(0,1), &\text{if} ~ \epsilon_2=\epsilon_3, \\
					0 ~ \text{for all} ~ \alpha\in(1,\infty), &\text{if} ~ \epsilon_2\neq\epsilon_3.
				\end{cases}
			\end{align*}
			Now, for the sign pattern $\epsilon=(\epsilon_1,\epsilon_2,\epsilon_3)$ with $\epsilon_1\neq\epsilon_2$ and $\epsilon_3\in\{\pm1\}$, the result follows from Remark~\ref{Remark_SSR_3}(iii) and the above case.
			
			\item \textbf{Case 2.} $d=4$: We claim that only the trivial powers entrywise preserve all $4\times4$ SR$(\epsilon)$ matrices. Embedding the $3\times3$ SR matrix from case 1 via padding by zeros, one can conclude that powers in the interval $(0,1)$ fail to preserve $4\times4$ SR$(\epsilon)$ matrices when $\epsilon_2=\epsilon_3$, while powers in the interval $(1,\infty)$ fail to preserve them when $\epsilon_2\neq\epsilon_3$. This result is summarized in the following table along with the examples needed to discard the remaining powers.
			\begin{table}[H]
				\begin{tabular}{|c|c|c|}
					\hline
					$\boldsymbol{\epsilon}$ & $\boldsymbol{\alpha : A^{\circ\alpha}}$ \textbf{is not SR}$\boldsymbol{(\epsilon)}$ & \textbf{To construct SR}$\boldsymbol{(\epsilon)}$ \textbf{matrices} $\boldsymbol{A}_{\scriptscriptstyle{\alpha}}$ \textbf{such that} \\ [0.5ex] 
					\hline\hline
					
					$\epsilon_2=\epsilon_3$ & $(0,1)$ & $\epsilon_4\det A^{\circ\alpha}_{\scriptscriptstyle{\alpha}}<0$ $\forall$ $\alpha\in(1,\infty)$ \\ [0.5ex] \hline 
					
					$\epsilon_2\neq\epsilon_3$ & $(1,\infty)$ & $\epsilon_4\det A^{\circ\alpha}_{\scriptscriptstyle{\alpha}}<0$ $\forall$ $\alpha\in(0,1)$ \\ [0.5ex] \hline
				\end{tabular}
				
				\caption{Table listing the cases for $4\times4$ SR$(\epsilon)$ matrices.}
				\label{Table_SR_4}
			\end{table}
			
			To construct such matrices, observe that for any $4\times4$ SR$(\epsilon)$ matrix $A$, the following identities hold:
			\[\det A = \det (AP_4) \quad \text{and} \quad \det A^{\circ\alpha} = \det (A^{\circ\alpha}P_4) = \det (AP_4)^{\circ\alpha}.\] 
			This implies that $\epsilon_4=\epsilon_4^{\prime}$, and if $\epsilon_4\det A^{\circ\alpha}<0$, then $\epsilon_4^{\prime}\det(AP_4)^{\circ\alpha}<0$, where $\epsilon^{\prime}$ denotes the sign pattern of $AP_4$. Consequently, it suffices to construct four $4\times4$ SR$(\epsilon)$ matrices whose sign patterns are not related via $P_4$, and which satisfy the determinant conditions listed in Table~\ref{Table_SR_4}. We provide these examples below.
			
			\textbf{Subcase 2.1.} $\epsilon_2=\epsilon_3$: We show that for every $\alpha>1$, there exists a $4\times4$ SR$(\epsilon)$ matrix $A$ such that $\epsilon_4\det A^{\circ\alpha}<0$.
			
			First, consider the case $\epsilon_1=\epsilon_4$. Define the matrix
			\begin{align*}
				A_1(t) := \begin{pmatrix}
					1 & 1 & 1 & 1 \\
					1 & 1+2t & 1+3t & 1+4t \\
					1 & 1+4t & 1+6t & 1+8t \\
					1 & 1+5t & 1+8t & 1+11t
				\end{pmatrix}, \quad t\geq0.
			\end{align*}
			One can verify that all $2\times 2$ minors have the form $ut+vt^2$, where $u>0, v\geq0$; all $3\times 3$ minors are of the form $ut^2$, where $u\geq0$; and $\det A_1(t)=0$. Thus, $A_1(t)$ is SR$(\epsilon)$ with $\epsilon_2=\epsilon_3$ and $\epsilon_1=\epsilon_4$ for all $t\geq0$. Carrying out a Taylor expansion of $\det A_1(t)^{\circ\alpha}$ at $t=0$ gives
			\begin{align*}
				\det A_1(t)^{\circ\alpha}= 2(\alpha^3-\alpha^4)t^4 + O(t^5) ~ \implies ~ \displaystyle{\lim_{t\to0^+}\frac{\det A_1(t)^{\circ\alpha}}{t^4}} = 2(\alpha^3-\alpha^4).
			\end{align*}
			Since $2(\alpha^3-\alpha^4)<0$ for all $\alpha>1$, it follows that
			\[\displaystyle{\lim_{t\to0^+}}\det A_1(t)^{\circ\alpha} < 0, \quad \forall ~ \alpha>1.\]
			Hence, for any given $\alpha\in(1,\infty)$, there exists small $t_{\alpha}>0$ such that $\det A_1(t_{\alpha})^{\circ\alpha}<0$. Since $\epsilon_1=\epsilon_4$, we have $\epsilon_4\det A_1(t_{\alpha})^{\circ\alpha}<0$.
			
			Now consider the case $\epsilon_1\neq\epsilon_4$, and define the matrix
			\begin{align*}
				A_2(t) := \begin{pmatrix}
					1 & 1 & 1 & 1 \\
					1 & 1+3t & 1+5t & 1+7t \\
					1 & 1+9t & 1+17t & 1+27t \\[2pt]
					1 & 1+11t & 1+23t & 1+\frac{119}{3}t
				\end{pmatrix}, \quad t\geq0.
			\end{align*}
			One can show that all $2\times 2$ minors are of the form $ut+vt^2$, where $u>0, v\geq0$; all $3\times 3$ minors are of the form $ut^2$, where $u>0$; and $\det A_2(t)=0$. Hence, $A_2(t)$ is SR$(\epsilon)$ with $\epsilon_2=\epsilon_3$ and $\epsilon_1\neq\epsilon_4$ for all $t\geq0$. Performing a Taylor expansion of $\det A_2(t)^{\circ\alpha}$ at $t=0$ yields
			\begin{align*}
				\det A_2(t)^{\circ\alpha}=-\frac{1084}{3}(\alpha^3-\alpha^4)t^4 + O(t^5) ~ \implies ~ \displaystyle{\lim_{t\to0^+}\frac{\det A_2(t)^{\circ\alpha}}{t^4}} = -\frac{1084}{3}(\alpha^3-\alpha^4),
			\end{align*}
			and this is positive for all $\alpha>1$. It follows that for any $\alpha>1$, we have $\det A_2(t_{\alpha})^{\circ\alpha}>0$ for sufficiently small $t_{\alpha}>0$. Moreover, as $\epsilon_1\neq\epsilon_4$, this implies $\epsilon_4\det A_2(t_{\alpha})^{\circ\alpha}<0$.
			
			\textbf{Subcase 2.2.} $\epsilon_2\neq\epsilon_3$: For any given $\alpha\in(0,1)$, we construct an SR$(\epsilon)$ matrix $A$ such that $\epsilon_4\det A^{\circ\alpha}<0$.
			
			Again, we first assume $\epsilon_1=\epsilon_4$, and define the matrix
			\begin{align*}
				A_3(t) := \begin{pmatrix}
					1 & 1 & 1 & 1 \\
					1 & 1+2t & 1+3t & 1+4t \\[2pt]
					1 & 1+3t & 1+\frac{9}{2}t & 1+6t \\[2pt]
					1 & 1+5t & 1+\frac{15}{2}t & 1+10t
				\end{pmatrix}, \quad t\geq0.
			\end{align*}
			Observe that all $2\times 2$ minors are of the form $ut$, where $u>0$; all $3\times 3$ minors are zero; and hence $\det A_3(t)=0$. Thus, $A_3(t)$ is SR$(\epsilon)$ with $\epsilon_2\neq\epsilon_3$ and $\epsilon_1=\epsilon_4$ for all $t\geq0$. A Taylor expansion of $\det A_3(t)^{\circ\alpha}$ at $t=0$ gives
			\begin{align*}
				\det A_3(t)^{\circ\alpha}= -\frac{45}{4}(2\alpha^3-5\alpha^4+4\alpha^5-\alpha^6)t^6 + O(t^7) ~ \implies ~ \displaystyle{\lim_{t\to0^+}\frac{\det A_3(t)^{\circ\alpha}}{t^6}} = -\frac{45}{4}p(\alpha),
			\end{align*}
			where $p(\alpha):= 2\alpha^3-5\alpha^4+4\alpha^5-\alpha^6=\alpha^3(\alpha-1)^2 (2-\alpha)$. That $p$ is positive on $(0,1)$ is straightforward from the given factorization. Thus, $-\frac{45}{4}p(\alpha)<0$ for all $\alpha\in(0,1)$. This implies that for any $\alpha\in(0,1)$, $\epsilon_4\det A_3(t_{\alpha})^{\circ\alpha}<0$ for small $t_{\alpha}>0$, as $\epsilon_1=\epsilon_4$.
			
			Next, suppose $\epsilon_1\neq\epsilon_4$, and define
			\begin{align*}
				A_4(t) := \begin{pmatrix}
					1 & 1 & 1 & 1 \\
					1 & 1+2t & 1+3t & 1+4t \\[2pt]
					1 & 1+4t & 1+\frac{17}{3}t & 1+\frac{22}{3}t \\[2pt]
					1 & 1+5t & 1+7t & 1+9t
				\end{pmatrix}, \quad 0\leq t\leq1.
			\end{align*}
			Note that all $2\times 2$ minors are either of the form $ut$ or $u(t-t^2)$, where $u>0$; all $3\times 3$ minors are of the form $-ut^2$, where $u\geq0$; and $\det A_4(t)=0$. This shows that $A_4(t)$ is SR$(\epsilon)$ with $\epsilon_2\neq\epsilon_3$ and $\epsilon_1\neq\epsilon_4$ for all $0\leq t\leq1$. Now, we carry out a Taylor expansion of $\det A_4(t)^{\circ\alpha}$ at $t=0$ to obtain
			\begin{align*}
				\det A_4(t)^{\circ\alpha}= \frac{2}{9}(\alpha^3-\alpha^4)t^4 + O(t^5) ~ \implies ~ \displaystyle{\lim_{t\to0^+}\frac{\det A_4(t)^{\circ\alpha}}{t^4}} = \frac{2}{9}(\alpha^3-\alpha^4).
			\end{align*}
			Since $\frac{2}{9}(\alpha^3-\alpha^4)>0$ for all $\alpha\in(0,1)$, it follows that for any given $\alpha\in(0,1)$, we can choose small $t_{\alpha}>0$ such that $\epsilon_4\det A_4(t_{\alpha})^{\circ\alpha}<0$, as $\epsilon_1\neq\epsilon_4$. This finishes the proof for $d=4$.
			
			\item \textbf{Case 3.} $d\geq5$: The result follows by embedding the above $3\times3$ SR$(\epsilon)$ matrix $A$ and $4\times4$ SR$(\epsilon)$ matrices $A_i(t)$ into a larger $d\times d$ matrix as follows: $A\oplus0_{(d-3)\times (d-3)}$ and $A_i(t)\oplus0_{(d-4)\times (d-4)}$ for $i=1,\ldots,4$. This completes the proof. \qedhere
		\end{itemize}
	\end{proof}
	
	Using Theorem~\ref{Theorem_SR_power_classification}, we now classify the power functions that preserve $m\times n$ SSR$(\epsilon)$ matrices under entrywise application. The proof also uses the following result, proved in our recent work~\cite{CY-SSR_Construction24}, which allows the embedding of SSR matrices into higher-dimensional SSR matrices via the addition of rows and columns.
	
	\begin{theorem}\cite[Theorem A]{CY-SSR_Construction24}\label{Theorem_line_insertion_SSR}
		Given integers $m,n\geq1$ and an $m\times n$ SSR matrix, it is possible to add a row/column to any of its borders such that the resulting matrix remains SSR. If minors of a larger size occur, they can be made either all positive or all negative.
	\end{theorem}
	
	\begin{cor}\label{Corollary_SSR_power_classification}
		Assume that $m,n\geq1$ are integers, and define $d:=\min\{m,n\}$. Let $\epsilon=(\epsilon_1,\ldots,\epsilon_d)$ be a given sign pattern with $\epsilon_1=1$. For $\alpha\in\mathbb{R}$, the following statements are equivalent.
		\begin{itemize}
			\item[(1)] The map $x\mapsto x^{\alpha}$ entrywise preserves all $m\times n$ SSR$(\epsilon)$ matrices.
			\item[(2)] The map $x\mapsto x^{\alpha}$ entrywise preserves all $d\times d$ SSR$(\epsilon)$ matrices.
			\item[(3)]
			\begin{itemize}
				\item[(a)] For $d=1:$ $\alpha\in\mathbb{R}$.
				\item[(b)] For $d=2:$ $\alpha\in(0,\infty)$.
				\item[(c)] For $d=3:$ $\alpha\in(0,1]$ if $\epsilon_2\neq\epsilon_3$, and $\alpha\in[1,\infty)$ if $\epsilon_2=\epsilon_3$.
				\item[(d)] For $d\geq4:$ $\alpha=1$.
			\end{itemize}
		\end{itemize}
	\end{cor}
	\begin{proof}
		The implication (1)$\implies$(2) holds by extending a $d\times d$ SSR$(\epsilon)$ matrix to an $m\times n$ SSR$(\epsilon)$ matrix by adding rows/columns to its borders (see Theorem~\ref{Theorem_line_insertion_SSR}). The converse (2)$\implies$(1) is immediate. We have shown (3)$\implies$(2) in Theorem~\ref{Theorem_SSR_3_powers} for $d=3$, and the remaining cases follow easily. It remains to prove (2)$\implies$(3). The case $d=1$ is trivial. For $d\geq2$, observe that taking $\alpha=0$ gives a matrix whose entries are all one. Consequently, any square submatrix of size at least $2\times2$ has zero determinant, and hence $\alpha=0$ must be excluded. Letting $d=2$, it is easy to see that for any $\alpha>0$, $x^{\alpha}$ entrywise preserves $d\times d$ SSR$(\epsilon)$ matrices. Now, choose $x_1,\ldots,x_4>0$ such that $x_1x_4\neq x_2x_3$. Then $A=\begin{pmatrix}
			x_1 & x_2 \\
			x_3 & x_4
		\end{pmatrix}$ is SSR, while the determinants of $A$ and $A^{\circ\alpha}$ have opposite signs whenever $\alpha<0$. Hence, we must have $\alpha>0$. For $d\geq3$, the condition $\alpha>0$ ensures that we can extend $x^{\alpha}$ continuously to $x=0$. By continuity and Theorem~\ref{Theorem_Gantmacher_Krein}, the function $x\mapsto x^{\alpha}$ preserves SR$(\epsilon)$ matrices. Now Theorem~\ref{Theorem_SR_power_classification} completely characterizes such power functions, concluding the proof.
	\end{proof}
	
	\subsection{General entrywise preservers}
	
	Thus far, we have characterized all exponents $\alpha$ for which the transformed $m\times n$ matrix $A^{\circ\alpha}$ is SR$(\epsilon)$/SSR$(\epsilon)$ whenever $A$ is SR$(\epsilon)$/SSR$(\epsilon)$, for every fixed $m,n\geq1$ and any given sign pattern $\epsilon$ with $\epsilon_1=1$. We now turn our attention to the more challenging question of characterizing all functions that preserve the sets of $m\times n$ SR$(\epsilon)$ and SSR$(\epsilon)$ matrices for each $m,n\geq2$. Our approach begins by classifying the continuous functions, showing that such functions must be power functions. Subsequently, we establish that any entrywise preserver of (strict) sign regularity must be continuous on the interval $(0,\infty)$. We require the following classical result due to Darboux.
	
	\begin{theorem}\cite[Darboux]{Darboux75}\label{Theorem_Darboux}
		Let $f:\mathbb{R}\to\mathbb{R}$ satisfies the Cauchy functional equation, i.e., $f(x+y) = f(x) + f(y)$ for all $x,y\in\mathbb{R}$, and let $f$ be continuous at a single point. Then $f(x)=f(1)x$ for all $x\in\mathbb{R}$.
	\end{theorem}
	
	\begin{theorem}\label{Theorem_SR_Fixed_1_continuous_power}
		Let $m,n\geq2$ be integers, and let $\epsilon = (\epsilon_1,\ldots,\epsilon_{\min\{m,n\}})$ be a given sign pattern with $\epsilon_1=1$. Suppose $f:[0,\infty)\to\mathbb{R}$ is continuous and entrywise preserves all $m\times n$ SR($\epsilon$) matrices. Then $f(x)=f(1)x^{\alpha}$ for some $\alpha\geq0$.
	\end{theorem}
	\begin{proof}
		We begin by considering $m=n=2$. Define
		\begin{align*}
			A(x,y):=\begin{pmatrix}
				x & xy \\
				1 & y
			\end{pmatrix},\hspace{.3cm} B(x,y):=\begin{pmatrix}
				xy & x \\
				y & 1
			\end{pmatrix}, \quad x,y\geq0.
		\end{align*}
		Clearly, both $A(x,y)$ and $B(x,y)$ are SR$(\epsilon)$. Since $f[-]$ preserves the sign regularity of these matrices, $\epsilon_2\det f[A(x,y)], \epsilon_2\det f[B(x,y)] \geq0$. These imply
		\begin{equation}\label{SR2_f_relation}
			f(x)f(y)=f(1)f(xy), \quad \forall ~ x,y\geq0.
		\end{equation}
		Since $\epsilon_1=1$, $f(1)\geq0$. We now consider two cases.
		\begin{itemize}[leftmargin=0pt,label={}]
			\item \textbf{Case 1.} $f(1)=0$: Setting $x=y\geq0$ in~\eqref{SR2_f_relation}, we deduce that $f(x)=0$. So $f\equiv0$ on $[0,\infty)$.
			
			\item \textbf{Case 2.} $f(1)>0$: We first show $f$ is positive on $(0,\infty)$. Note that $f(x)\geq0$ for all $x>0$, which follows by considering the $2\times2$ SR$(\epsilon)$ matrix $(x)_{1\times1}\oplus0_{1\times1}$. Suppose, for contradiction, that $f(x_0)=0$ for some $x_0>0$. From equation~\eqref{SR2_f_relation} with $x=x_0$, $y=\frac{1}{x_0}$, we obtain 
			\[0=f(x_0)f\left(\frac{1}{x_0}\right)= f(1)^2 >0,\] 
			a contradiction. Hence, $f(x)>0$ for all $x\in(0,\infty)$. Now define the functions
			\begin{align}\label{SR2_g_h_define}
				g(x):=\frac{f(x)}{f(1)}, ~ x>0, \qquad h(y):=\log g(e^y), ~ y\in\mathbb{R}.
			\end{align}
			Since $f$ is positive and continuous on $(0,\infty)$, the function $h$ is well-defined and continuous. Using~\eqref{SR2_f_relation}~and~\eqref{SR2_g_h_define}, we have the following relations:
			\begin{align*}
				g(xy) = g(x)g(y), ~\forall ~ x,y>0, \qquad
				h(a+b) = h(a)+h(b), ~\forall ~ a,b\in\mathbb{R}.
			\end{align*}
			Since $h$ is continuous and satisfies the Cauchy functional equation, by Theorem~\ref{Theorem_Darboux}, $h(y)=h(1)y$ for all $y\in\mathbb{R}$. Thus
			\begin{align*}
				\log g(e^y)=h(1)y, \quad \forall ~ y\in\mathbb{R}.
			\end{align*}
			Taking $e^y=x$ above and using \eqref{SR2_g_h_define}, we obtain the following for all $x>0$:
			\begin{align*}
				\log g(x) = h(1)\log x
				\implies f(x) = f(1)x^{h(1)}.
			\end{align*}
			Finally, continuity of $f$ at $x=0$ implies $h(1)\geq0$. Therefore, $f(x)=f(1)x^{\alpha}$ for all $x>0$, with $\alpha:=h(1)\geq0$. We now determine the value of $f(0)$ using the continuity of $f$.
			\begin{itemize}
				\item[(i)] If $\alpha=0$, then $f(x)=f(1)$ for all $x>0$. To ensure continuity of $f$ at $x=0$, we must have $f(0)=f(1)$.
				
				\item[(ii)] If $\alpha>0$, then $\displaystyle{\lim_{x\to0^+}}f(x)=0$, so $f(0)=0$. Thus, $f(x)=f(1)x^{\alpha}$ for all $x\geq0$.
			\end{itemize}
		\end{itemize}
		To extend this result to arbitrary $m,n\geq2$, note that one can embed $A(x,y)$ and $B(x,y)$ into $m\times n$ SR$(\epsilon)$ matrices by padding them with zeros. Thus, the hypotheses imply that $f[-]$ preserves the sign regularity of the submatrices $A(x,y)$ and $B(x,y)$, and the result follows.
	\end{proof}
	
	Our next objective is to prove that for any fixed integers $m,n\geq2$, every entrywise map preserving $m\times n$ SR$(\epsilon)$ matrices with non-negative entries must be continuous on $(0,\infty)$. To show this, we first present several intermediate results.
	
	\begin{theorem}\label{Theorem_SR_Fixed_1_nonnegative_nondecreasing}
		Let $f:[0,\infty)\to\mathbb{R}$ and $\epsilon=(\epsilon_1,\epsilon_2)$ be a given sign pattern with $\epsilon_1=1$. Then the following statements are equivalent.
		\begin{itemize}
			\item[(1)] $f[-]$ preserves the set of $2\times2$ SR($\epsilon$) matrices with $\epsilon_1=\epsilon_2$ of the form
			\begin{align*}
				\begin{pmatrix}
					a & b \\
					b & c
				\end{pmatrix}, ~ \text{where} ~ a,b,c\geq0 ~ \text{and} ~ 0\leq b\leq\sqrt{ac}.
			\end{align*}
			
			\item[(2)] $f[-]$ preserves the set of $2\times2$ SR($\epsilon$) matrices with $\epsilon_1\neq\epsilon_2$ of the form
			\begin{align*}
				\begin{pmatrix}
					b & a \\
					c & b
				\end{pmatrix}, ~ \text{where} ~ a,b,c\geq0 ~ \text{and} ~ 0\leq b\leq\sqrt{ac}.
			\end{align*}
			
			\item[(3)] $f$ is non-negative, non-decreasing, and multiplicatively mid-convex on $[0,\infty)$; that is,
			\begin{align*}
				f(\sqrt{xy})\leq\sqrt{f(x)f(y)}, \quad \forall ~ x,y\geq0.
			\end{align*}
			Moreover, $f$ is either identically zero or never zero on $(0,\infty)$.
		\end{itemize}
		This theorem also holds if the domain $[0,\infty)$ is replaced by $(0,\infty)$.
	\end{theorem}
	\begin{proof}
		We establish the cyclic chain of implications: (1)$\implies$(2)$\implies$(3)$\implies$(1). Note that (1)$\implies$(2) is immediate by Remark~\ref{Remark_SSR_3}(iii). To show (2)$\implies$(3), let $\epsilon_1\neq\epsilon_2$ and assume that $f[-]$ preserves all $2\times2$ SR$(\epsilon)$ matrices of the form $\begin{pmatrix}
			b & a \\
			c & b
		\end{pmatrix}$. Consider the matrix
		\begin{align*}
			A_1=\begin{pmatrix}
				b & a \\
				a & b
			\end{pmatrix}, \quad 0\leq b\leq a.
		\end{align*} 
		Note that $A_1$ is SR with $\epsilon_1\neq\epsilon_2$. By hypotheses, $f[A_1]$ is SR$(\epsilon)$ and the determinant of $f[A_1]$ yields
		\begin{align*}
			0\leq f(b)\leq f(a), ~ \text{whenever} ~ 0\leq b\leq a.
		\end{align*}
		This shows that $f$ is non-decreasing and non-negative on $[0,\infty)$. Next, to show that $f$ is multiplicatively mid-convex, consider the matrix
		\begin{align*}
			A_2 = \begin{pmatrix}
				\sqrt{ac} & a \\
				c & \sqrt{ac}
			\end{pmatrix}, \quad a,c\geq0.
		\end{align*}
		Then $A_2$ is SR$(\epsilon)$ and so is $f[A_2]$. Thus, the determinant of $f[A_2]$ gives \[f(\sqrt{ac})\leq\sqrt{f(a)f(c)}, \quad \forall ~ a,c\geq0.\]
		Thus, $f$ is multiplicatively mid-convex on $[0,\infty)$.
		
		Next, we claim that $f$ is either identically zero or never zero on $(0,\infty)$. To show this, let $I^+ = (0,\rho)$ for some $0<\rho\leq\infty$. Suppose that $f(x)=0$ for some $x\in I^+$. Since $f$ is non-negative and non-decreasing on $I^+$, $f\equiv0$ on $(0,x]$. It remains to show that $f(y)=0$ for all $x<y<\rho$. Let $y\in(x,\rho)$ and choose large $n$ such that $y\left(\frac{y}{x}\right)^{\frac{1}{n}}<\rho$. Define
		\begin{align*}
			x_k:=x\left(\frac{y}{x}\right)^{\frac{k}{n}} ~ \text{for} ~ k=0,1,\ldots,n+1.
		\end{align*}
		Then, $x_k\in I^+$ for all $k$ and $x_0<x_1<\cdots<x_{n+1}$. Now, for each $i=0,1,\ldots,n-1$, define an SR$(\epsilon)$ matrix
		\begin{align*}
			A_i:=\begin{pmatrix}
				x_{i+1} & x_{i} \\
				x_{i+2} & x_{i+1}
			\end{pmatrix}.
		\end{align*}
		By assumption, $f[A_i]$ is SR$(\epsilon)$ with $\epsilon_1\neq\epsilon_2$ for all $0\leq i\leq n-1$. Taking their determinant yields
		\begin{align*}
			0\leq f(x_{i+1}) \leq\sqrt{f(x_i)f(x_{i+2})}, \quad 0\leq i\leq n-1.
		\end{align*}
		Using induction and the fact that $f(x_0)=f(x)=0$, it follows that $f(x_i)=0$ for all $0\leq i\leq n$. In particular, $f(x_{n})=f(y)=0$. Thus $f\equiv0$ on $I^+$.
		
		Finally, we show that (3)$\implies$(1). Let $\epsilon_1=\epsilon_2$ and consider the following SR$(\epsilon)$ matrix
		\begin{align*}
			A_3 = \begin{pmatrix}
				a & b \\
				b & c
			\end{pmatrix}.
		\end{align*}
		This gives the following conditions on their entries: $a,b,c\geq0 ~ \text{and} ~ 0\leq b\leq\sqrt{ac}$. Since $f$ is non-negative, non-decreasing, and multiplicatively mid-convex on $[0,\infty)$, we have
		\[f(a), f(b), f(c)\geq0; \qquad f(b)\leq f(\sqrt{ac})\leq\sqrt{f(a)f(c)}.\] 
		This shows that $f[A_3]$ is SR$(\epsilon)$.
	\end{proof}
	
	The preceding result shows that any function that entrywise preserves the set of $2\times2$ SR$(\epsilon)$ matrices with $\epsilon_1=1$ must be non-negative, non-decreasing, and multiplicatively mid-convex on $[0,\infty)$. Moreover, such a function is either identically zero on $(0,\infty)$ or positive on $(0,\infty)$. A classical result of Ostrowski~\cite{Ostrowski29} shows that any such function is continuous on $(0,\infty)$; for completeness, we recall this result below.
	
	\begin{theorem}\cite[Ostrowski]{Ostrowski29} \label{Theorem_SR_Fixed_1_continuous}
		Let $f:(0,\infty)\to\mathbb{R}$ be a function that is positive, non-decreasing, and multiplicatively mid-convex. Then $f$ is continuous.
	\end{theorem}
	
	Before proceeding to the proof of our first main result, we recall a few definitions and a key result.
	
	\begin{defn}
		Let $A=(a_{ij})$ be an $m\times n$ matrix.
		\begin{itemize}
			\item[(i)] We say that $A$ is in \textit{double echelon form} if the following conditions hold:
			\begin{itemize}
				\item[$\bullet$] Each row of $A$ is of one of the following forms:
				\[(*,\ldots,*), ~ (0,\ldots,0,*,\ldots,*), ~ (*,\ldots,*,0,\ldots,0), ~ \text{or} ~ (0,\ldots,0,*,\ldots,*,0,\ldots,0).\]
				Here $*$ denotes a non-zero entry of $A$.
					
				\item[$\bullet$] For each $i=1,\ldots,m-1$, the first and last non-zero entries in row $i+1$ are not to the left of those in row $i$, respectively.
			\end{itemize}
				
			\item[(ii)] An $m\times n$ \textit{$(0,+)$-pattern} is an $m\times n$ array whose entries belong to $\{0,+\}$.
				
			\item[(iii)] Let $S=(s_{ij})$ be an $m\times n$ $(0,+)$-pattern. A \textit{realization} of $S$ is an $m\times n$ matrix $A$ such that
			\begin{align*}
				a_{ij}>0 ~ \text{if} ~ s_{ij}=+, ~ \text{and} ~ a_{ij}=0 ~ \text{if} ~ s_{ij}=0.
			\end{align*}
				
			\item[(iv)] A $(0,+)$-pattern \textit{allows property $P$} if there exists a realization of $(0,+)$-pattern with property $P$.
		\end{itemize}
	\end{defn}
		
	We next recall a key result that identifies the precise locations of zero entries in matrices whose minors of size up to $2\times2$ are all non-negative. Such matrices are called \textit{totally non-negative matrices of order $2$} (TN$_2$).
		
	\begin{theorem}\cite[Theorem 1.6.4]{fallat-john}\label{TP2_NE_SW} 
		Let $m,n\geq2$ and let $S$ be an $m\times n$ $(0,+)$-pattern with no zero rows or columns. Then $S$ allows a TN$_2$ matrix if and only if $S$ is a double echelon pattern.
	\end{theorem}
	
	Equipped with the preceding results, we are now ready to prove Theorem~\ref{Theorem_SR_Fixed_Classification}.
	
	\begin{proof}[Proof of Theorem~\ref{Theorem_SR_Fixed_Classification}]
		We show the equivalences (1)$\iff$(2) and (2)$\iff$(3), which together complete the proof. The implication (1)$\implies$(2) follows by embedding any $d\times d$ SR$(\epsilon)$ matrix into an $m\times n$ SR$(\epsilon)$ matrix by adding zero rows or columns. The converse implication, (2)$\implies$(1), is immediate.
		
		We next prove (3)$\implies$(2).	Let $\epsilon=(\epsilon_1,\ldots,\epsilon_d)$ with $\epsilon_1=1$. The case $d=1$ is trivial. For $d\geq2$, the classification of power functions that preserve $d\times d$ SR$(\epsilon)$ matrices is given in Theorem~\ref{Theorem_SR_power_classification}.
		
		Next, we determine the sign patterns $\epsilon$ for which the signum function entrywise preserves $d\times d$ SR$(\epsilon)$ matrices for all $d\geq2$. Let $c>0$ and define $f(x):=c\sgn(x)$ for $x\geq0$. This function clearly preserves all $2\times 2$ SR$(\epsilon)$ matrices. However, it fails to preserve $d\times d$ SR$(\epsilon)$ matrices when $\epsilon_2=\epsilon_3$ for all $d\geq3$. For instance, let $d=3$ and take the matrix $A$ equal to
		\begin{align*}
			\text{either} ~ \begin{pmatrix}
				3 & 1 & 0 \\
				1 & 1 & 1 \\
				0 & 2 & 4 
			\end{pmatrix} ~ \text{or} ~ \begin{pmatrix}
				3 & 1 & 0 \\
				1 & 1 & 1 \\
				0 & 2 & 4 
			\end{pmatrix}\cdot P_3.
		\end{align*}
		Then $A$ is SR$(\epsilon)$ with $\epsilon_2=\epsilon_3$, but $\epsilon_3\det f[A]<0$. Extending $A$ by padding with zeros produces counterexamples for all $d\geq4$, showing that the signum function fails to preserve SR$(\epsilon)$ matrices whenever $\epsilon_2=\epsilon_3$.
		
		Let $\epsilon_2\neq\epsilon_3$. We claim that $f[-]$ preserves all $d\times d$ SR$(\epsilon)$ matrices for $d\geq3$. We first consider $d=3$. Let $A$ be any $3\times3$ SR$(\epsilon)$ matrix. If all entries of $A$ are positive, then $\det f[A]=0$, and the matrix $f[A]$ is SR$(\epsilon)$. Let $A$ be a matrix with at least one zero entry and let $\det f[A]\neq0$. We claim that $\epsilon_3\det f[A]>0$. Without loss of generality assume that $\epsilon_2=1$, otherwise one can replace $A$ by $AP_3$. Since $A$ is TN$_2$, we apply Theorem~\ref{TP2_NE_SW} to derive the following constraints on the entries $a_{ij}$ of $A$.
		\begin{itemize}
			\item[(i)] If $a_{11}=0$, then either the first row or first column of $A$ is entirely zero, leading to $\det f[A]=0$. Thus, $a_{11}\neq0$. Similarly, $a_{33}\neq0$.
			
			\item[(ii)] If $a_{12}=0$, then $a_{13}=0$ (otherwise, the second column of $A$ is zero, and thus $\det f[A]=0$), and the submatrix
			\begin{align}\label{2by2submatrix}
				\begin{pmatrix}
					a_{22} & a_{23} \\
					a_{32} & a_{33}
				\end{pmatrix}
			\end{align} 
			of $A$ must contain at least one zero entry since $\det f[A]\neq0$. If $a_{22}=0$, then either $a_{23}=0$ or $a_{21}=a_{31}=a_{32}=0$. In both cases, $\det f[A]=0$. If $a_{23}=0$, then $\epsilon_3\det A<0$, which is not true. Hence, we must have $a_{22}\neq0$ and $a_{23}\neq0$, and the only possible candidate for a zero entry is $a_{32}$. However, this configuration again leads to $\epsilon_3\det A<0$, which is a contradiction. Therefore, $a_{12}\neq0$.
			
			\item[(iii)] If $a_{21}=0$, then $a_{31}=0$, and the $2\times2$ submatrix of $A$ in \eqref{2by2submatrix} must contain a zero entry since $\det f[A]\neq0$. If $a_{22}=0$, then $a_{32}=0$, which gives $\det f[A]=0$. Hence, $a_{22}\neq0$. If either $a_{23}=0$ or $a_{32}=0$, both cases yield $\det A>0$, contradicting $\epsilon_2\neq\epsilon_3$. Therefore, $a_{21}\neq0$, and since $a_{12}\neq0$, we also have $a_{22}\neq0$.
			
			\item[(iv)] If $a_{23}=0$, then $a_{13}=0$ and $\det f[A]=0$, a contradiction. Thus, $a_{23}\neq0$. Similarly, if $a_{32}=0$, then $a_{31}=0$, and again $f[A]$ becomes singular.
		\end{itemize}
		Therefore, all entries of $A$ must be non-zero except possibly $a_{13}$ and $a_{31}$. If exactly one of these is zero, then $\det f[A]=0$. Thus both $a_{13}=0$ and $a_{31}=0$. Then $\det f[A]\neq0$, and $\epsilon_3\det f[A]>0$. This shows the claim for $d=3$.
		
		Let $d=4$ and $A$ be any $4\times4$ SR$(\epsilon)$ matrix. We claim that $f[A]$ has rank at most $3$, and therefore $\det f[A]=0$, implying that $f[-]$ preserves all such SR$(\epsilon)$ matrices. The same reasoning applies to all $d\geq4$, completing the argument that the signum function entrywise preserves $d\times d$ SR$(\epsilon)$ matrices whenever $\epsilon_2\neq\epsilon_3$.
		
		Using Remark~\ref{Remark_SSR_3}(iii), without loss of generality assume that $\epsilon_2=1$. Suppose for contradiction that $\det f[A]\neq0$. Then, $f[A]$ must contain a $3\times3$ submatrix of the following form -- this follows from the $d=3$ case:
		\begin{align}\label{Signum_nonzero_3x3}
			\begin{pmatrix}
				c & c & 0 \\
				c & c & c \\
				0 & c & c
			\end{pmatrix}.
		\end{align}
		Since $f[A]$ is SR$(\epsilon)$ and $\epsilon_1=\epsilon_2=1$, $f[A]$ is TN$_2$. Using Theorem~\ref{TP2_NE_SW} and computing the determinant of $f[A]=(f(a_{ij}))$ along the first row, we make the following observations:
		\begin{itemize}
			\item[(i)] If $f(a_{11})=0$, then either the first row or first column of $f[A]$ must be entirely zero, leading to $\det f[A]=0$. Thus $f(a_{11})\neq0$.
			
			\item[(ii)] If $f(a_{12})=0$, since $\det f[A]\neq0$, by Theorem~\ref{TP2_NE_SW}, we must have $f(a_{13})=f(a_{14})=0$. For $\det f[A]\neq0$, the matrix in~\eqref{Signum_nonzero_3x3} must be the principal submatrix of $f[A]$ indexed by rows and columns $\{2,3,4\}$. This, in turn, implies that the $3\times3$ submatrix of $f[A]$ formed by rows $\{1,2,3\}$ and columns $\{1,3,4\}$ is upper triangular with non-zero diagonal entries, and therefore has a positive determinant, contradicting the assumption $\epsilon_2\neq\epsilon_3$. Hence, $f(a_{12})\neq0$.
			
			\item[(iii)] If $f(a_{13})=0$, by a similar argument as above $f(a_{14})=0$. Since $\det f[A]\neq0$, the invertible $3\times3$ matrix in~\eqref{Signum_nonzero_3x3} must be a submatrix of $f[A]$. Hence, $f[A]$ takes one of the forms:
			\begin{align*}
				\begin{pmatrix}
					c & c & 0 & 0 \\
					* & c & c & 0 \\
					* & c & c & c \\
					* & 0 & c & c
				\end{pmatrix} \quad \text{or} \quad \begin{pmatrix}
					c & c & 0 & 0 \\
					c & * & c & 0 \\
					c & * & c & c \\
					0 & * & c & c
				\end{pmatrix}, \quad \text{where}~*\in\{0,c\}.
			\end{align*}
			In both configurations, the $3\times3$ submatrix occupying the upper-right corner of $A$ has a positive determinant. This contradicts the assumption $\epsilon_2\neq\epsilon_3$, so we conclude that $f(a_{13})\neq0$.
			
			\item[(iv)] If $f(a_{14})=0$, then by adapting the same argument as in (iii) above, we conclude that the structure of $f[A]$ must be one of the following:
			\begin{align*}
				\begin{pmatrix}
					c & c & c & 0 \\
					* & c & c & 0 \\
					* & c & c & c \\
					* & 0 & c & c
				\end{pmatrix} \quad \text{or} \quad \begin{pmatrix}
					c & c & c & 0 \\
					c & * & c & 0 \\
					c & * & c & c \\
					0 & * & c & c
				\end{pmatrix} \quad \text{or} \quad \begin{pmatrix}
					c & c & c & 0 \\
					c & c & * & 0 \\
					c & c & * & c \\
					0 & c & * & c
				\end{pmatrix}, \quad \text{where}~*\in\{0,c\}.
			\end{align*}
			Since $f[A]$ is TN$_2$, applying Theorem~\ref{TP2_NE_SW} to complete these matrices, we obtain
			\begin{align*}
				\begin{pmatrix}
					c & c & c & 0 \\
					* & c & c & 0 \\
					* & c & c & c \\
					0 & 0 & c & c
				\end{pmatrix} \quad \text{or} \quad \begin{pmatrix}
					c & c & c & 0 \\
					c & c & c & 0 \\
					c & c & c & c \\
					0 & * & c & c
				\end{pmatrix} \quad \text{or} \quad \begin{pmatrix}
					c & c & c & 0 \\
					c & c & c & 0 \\
					c & c & c & c \\
					0 & c & c & c
				\end{pmatrix}, \quad \text{where}~*\in\{0,c\}.
			\end{align*}
			The final matrix has two identical columns, so $\det f[A]=0$. In the second matrix, regardless of whether $*=0$ or $c$, two columns again become identical, implying $\det f[A]=0$. In the first case, if $f(a_{21})=0$, then $f(a_{31})=0$, which forces the $3\times3$ submatrix of $A$ indexed by rows $\{1,2,3\}$ and columns $\{1,2,4\}$ to have a positive determinant, a contradiction. Hence, $f(a_{21})\neq0$. Similarly, if $f(a_{31})=0$, the lower-left $3\times3$ submatrix of $A$ becomes upper triangular with positive diagonal entries, again violating the required sign condition. Thus $f(a_{31})\neq0$, and $\det f[A]=0$, a contradiction.
		\end{itemize}
		Thus, $f(a_{11})=f(a_{12})=f(a_{13})=f(a_{14})\neq0$. Since $\det f[A]\neq0$, $f[A]$ contains a $3\times3$ submatrix of the form~$\eqref{Signum_nonzero_3x3}$ below the first row. One can verify that the matrix $A$ corresponding to every possible configuration of $f[A]$ contains at least one $2\times2$ negative minor, which contradicts the assumption. Thus, $\det f[A]=0$. This shows that the signum function, applied entrywise, preserves all $d\times d$ SR$(\epsilon)$ matrices whenever $\epsilon_2\neq\epsilon_3$ for $d\geq4$.
		
		Finally, we prove (2)$\implies$(3). The case $d=1$ is trivial. For $d\geq2$, we assert that the only functions that entrywise preserve $d\times d$ SR$(\epsilon)$ matrices are either power functions or the signum function. This combined with Theorem~\ref{Theorem_SR_power_classification} and the proof of (3)$\implies$(2), completes the proof.
		
		Consider first the case $d=2$. By Theorem~\ref{Theorem_SR_Fixed_1_nonnegative_nondecreasing}, the function $f$ must be non-negative, non-decreasing, and multiplicatively mid-convex on $[0,\infty)$. Moreover, Theorem~\ref{Theorem_SR_Fixed_1_continuous} ensures that $f$ is continuous on $(0,\infty)$. Following a similar argument to that in Theorem~\ref{Theorem_SR_Fixed_1_continuous_power}, we conclude that
		\begin{align}\label{Theorem_A_fx_fy=f1_fxy}
			f(x)f(y) = f(1)f(xy), \quad \forall ~ x,y\geq0.
		\end{align}
		Moreover, either $f\equiv0$ on $[0,\infty)$ or $f(x)=f(1)x^{\alpha}$ with $f(1)>0$, $\alpha\geq0$, and all $x>0$. Let $f$ be a non-zero function. Then $f(x)=f(1)x^{\alpha}$ as above. We now determine the value of $f(0)$, noting that $f(0)\geq0$ since $\epsilon_1=1$.
		\begin{itemize}[leftmargin=0pt,label={}]
			\item \textbf{Case 1.} $f(0)>0$: Substituting $x=0$ and $y>0$ into \eqref{Theorem_A_fx_fy=f1_fxy} yields $f(y) = f(1)$ for all $y>0$, implying $\alpha=0$. Substituting $x=y=0$ gives $f(0)=f(1)$. Hence, $f(x)=f(1)$ for all $x\geq0$.
			
			\item \textbf{Case 2.} $f(0)=0$: If $\alpha>0$, then $f(x)=f(1)x^{\alpha}$ for all $x\geq0$, and $f$ is continuous on $[0,\infty)$. The permissible values of $\alpha$ have already been classified in Theorem~\ref{Theorem_SR_power_classification}. If $\alpha=0$, then $f(x)=f(1)$ for all $x>0$, and with $f(0)=0$, we obtain the (scaled) signum function. The sign patterns for which the signum function preserves SR$(\epsilon)$ matrices are classified in the preceding implication.
		\end{itemize}
		For $d\geq3$, by embedding each $2\times2$ SR matrix considered in Theorem~\ref{Theorem_SR_Fixed_1_nonnegative_nondecreasing} into a $d\times d$ SR$(\epsilon)$ matrix by padding with zeros, it follows that $f[-]$ preserves all $2\times2$ SR$(\epsilon)$ matrices. Thus, $f$ is non-negative, non-decreasing, and multiplicatively mid-convex on $[0,\infty)$. Proceeding as in the proof for the case $d=2$, the only functions preserving SR$(\epsilon)$ matrices entrywise in higher dimensions are power functions and the signum function. This finishes the proof.
	\end{proof}
	
	We next prove Theorem~\ref{Theorem_SSR_Fixed_Classification} using the following lemma.
	\begin{lemma}\label{Lemma_SSR_Fixed_positive_increasing_mmc}
		Let $f:(0,\infty)\to\mathbb{R}$ and $\epsilon = (\epsilon_1,\epsilon_2)$ be a given sign pattern with $\epsilon_1=1$. Then the following statements are equivalent.
		\begin{itemize}
			\item[(1)] $f[-]$ preserves the set of $2\times2$ SSR$(\epsilon)$ matrices with $\epsilon_1=\epsilon_2$ of the form
			\begin{align*}
				\begin{pmatrix}
					a & b \\
					b & c
				\end{pmatrix}, ~ \text{where} ~ a,b,c >0 ~ \text{and} ~ 0<b<\sqrt{ac}.
			\end{align*}
			
			\item[(2)] $f[-]$ preserves the set of $2\times2$ SSR$(\epsilon)$ matrices with $\epsilon_1\neq\epsilon_2$ of the form
			\begin{align*}
				\begin{pmatrix}
					b & a \\
					c & b
				\end{pmatrix}, ~ \text{where} ~ a,b,c >0 ~ \text{and} ~ 0<b<\sqrt{ac}.
			\end{align*}
			
			\item[(3)] $f$ is positive, strictly increasing, and multiplicatively mid-convex on $(0,\infty)$. Moreover, $f$ is continuous on $(0,\infty)$.
		\end{itemize}
	\end{lemma}
	\begin{proof}
		That (1)$\implies$(2) is trivial. The implication (3)$\implies$(1) follows similarly to the proof of Theorem~\ref{Theorem_SR_Fixed_1_nonnegative_nondecreasing}. We now show (2)$\implies$(3) to complete the proof. Let $\epsilon_1\neq\epsilon_2$ and suppose that $f[-]$ preserves all $2\times2$ SSR$(\epsilon)$ matrices of the form $\begin{pmatrix}
			b & a \\
			c &b
		\end{pmatrix}$. Consider the matrix
		\begin{align*}
			A_1(x,y):=\begin{pmatrix}
				x & y \\
				y & x
			\end{pmatrix}, \quad 0<x<y<\infty.
		\end{align*}
		Then $A_1(x,y)$ is SSR with $\epsilon_1\neq\epsilon_2$, and so is the matrix $f[A_1(x,y)]$. The determinant of $f[A_1(x,y)]$ gives $f(y)>f(x)$ whenever $y>x$, which implies that $f$ is strictly increasing on $(0,\infty)$. To show positivity of $f$, consider the following matrix. 
		\begin{align*}
			A_2(x,\delta):=\begin{pmatrix}
				x & x+\delta \\
				x & x
			\end{pmatrix}, \quad x,\delta>0.
		\end{align*}
		Observe that $A_2(x,\delta)$ is SSR$(\epsilon)$ with $\epsilon_1\neq\epsilon_2$. By hypotheses, $f[A_2(x,\delta)]$ is SSR$(\epsilon)$ and its $1\times1$ minors give $f(x)>0$ for all $x>0$.
		
		Since $f$ is strictly increasing on $(0,\infty)$, it can have at most countably many discontinuities, all of which are jump discontinuities. For $x>0$, define the function $f^+(x):=\displaystyle{\lim_{y\to x^+}}f(y).$ Then $f^+(x)\geq f(x)$ for all $x$, and $f^+(x)=f(x)$ at all points of right continuity and has the same jumps as $f$. To establish continuity of $f$, it suffices to show $f^+(x)$ is continuous. Consider the following matrix
		\begin{align*}
			A_3(x,y,\delta):=\begin{pmatrix}
				\sqrt{xy} + \delta & x+\delta \\
				y + \delta & \sqrt{xy} + \delta
			\end{pmatrix}, \quad  x\neq y >0 ~ \text{and} ~ \delta>0.
		\end{align*}
		This matrix is SSR with $\epsilon_1\neq\epsilon_2$. By hypotheses, $\epsilon_2\det f[A_3(x,y,\delta)]>0$, and thus
		\[f(x+\delta)f(y+\delta)> f(\sqrt{xy}+\delta)^2, \quad \forall ~ x\neq y>0 ~ \text{and} ~ \delta>0.\]
		Taking $\delta\to0^+$, we obtain
		\begin{align*}
			f^+(x)f^+(y) \geq f^+(\sqrt{xy})^2, \quad \forall ~ x,y>0.
		\end{align*}
		This shows that $f^+$ is positive, strictly increasing, and multiplicatively mid-convex on $(0,\infty)$. Hence, by Theorem~\ref{Theorem_SR_Fixed_1_continuous}, $f^+$ is continuous on $(0,\infty)$, which implies that $f$ itself is continuous and multiplicatively mid-convex on $(0,\infty)$.
	\end{proof}
	
	We are now prepared to prove Theorem~\ref{Theorem_SSR_Fixed_Classification}.
	
	\begin{proof}[Proof of Theorem \ref{Theorem_SSR_Fixed_Classification}]
		We establish the same equivalences as those in Theorem~\ref{Theorem_SR_Fixed_Classification} to complete the proof. First, observe that (1)$\implies$(2) follows from Theorem~\ref{Theorem_line_insertion_SSR}, which ensures that any $d\times d$ SSR$(\epsilon)$ matrix can be extended to an $m\times n$ SSR$(\epsilon)$ matrix by inserting rows and columns. The converse, (2)$\implies$(1), is straightforward. To show (3)$\implies$(2), we note that the case $d=1$ is immediate, while the remaining cases follow from Corollary~\ref{Corollary_SSR_power_classification}. 
		
		It therefore remains to show that (2)$\implies$(3). For $d=1$, the result holds trivially. Let $d=2$. By Lemma~\ref{Lemma_SSR_Fixed_positive_increasing_mmc}, $f$ is positive, strictly increasing, and continuous on $(0,\infty)$. We claim that $f(x) = f(1)x^{\alpha}$ for all $x>0$, with $\alpha>0$. Let $\epsilon_1=\epsilon_2$ and define
		\begin{align*}
			A(x,y,\delta) := \begin{pmatrix}
				x & xy \\
				1-\delta & y
			\end{pmatrix}, \quad B(x,y,\delta) := \begin{pmatrix}
				xy & y \\
				x & 1+\delta
			\end{pmatrix}, \quad x,y>0 ~ \text{and} ~ 0<\delta<1.
		\end{align*}
		Both $A(x,y,\delta)$ and $B(x,y,\delta)$ are SSR$(\epsilon)$ matrices, so $\epsilon_2\det f[A(x,y,\delta)], \epsilon_2\det f[B(x,y,\delta)]>0$. These inequalities yield
		\begin{align*}
			f(x)f(y) > f(xy)f(1-\delta) ~ \text{and} ~
			f(xy)f(1+\delta) > f(x)f(y), \quad \forall ~ x,y>0  ~ \text{and} ~ 0<\delta<1.
		\end{align*}
		Taking the limit as $\delta\to0^+$ and applying continuity of $f$, we obtain
		\begin{align*}
			f(x)f(y)=f(1)f(xy), \quad \forall ~ x,y>0.
		\end{align*}
		Hence, the function $\frac{f(x)}{f(1)}$ is continuous, positive, strictly increasing, and multiplicative on $(0,\infty)$. Following a proof similar to case 2 of Theorem~\ref{Theorem_SR_Fixed_1_continuous_power}, it follows that $f(x) = f(1)x^{\alpha}$ for all $x>0$, with $\alpha>0$. The same conclusion holds when $\epsilon_1\neq\epsilon_2$ by applying the above arguments to the matrices $A(x,y,\delta)P_2$ and $B(x,y,\delta)P_2$.
		
		For $d\geq3$, we can extend the above $2\times2$ SSR matrices to $d\times d$ SSR$(\epsilon)$ matrices using Theorem~\ref{Theorem_line_insertion_SSR}. Now $f[-]$ preserves all $2\times 2$ SSR$(\epsilon)$ matrices, and the conclusion follows from the case $d=2$. The admissible exponents $\alpha$ are classified in Corollary~\ref{Corollary_SSR_power_classification}, which concludes the proof.
	\end{proof}
	
	We next apply Theorems~\ref{Theorem_SR_Fixed_Classification}~and~\ref{Theorem_SSR_Fixed_Classification} to classify all entrywise functions that preserve (strict) sign regularity with non-positive (negative) entries, for every fixed sign pattern and in each fixed dimension.
	
	\begin{cor}\label{Corollary_SR_epsilon_1=-1}
		Let $m,n\geq1$ be integers, and let $\epsilon=(\epsilon_1,\ldots,\epsilon_d)$ be a sign pattern with $\epsilon_1=-1$, where $d:=\min\{m,n\}$. Let $g:(-\infty,0]\to\mathbb{R}$ be a function. The following are equivalent.
		\begin{itemize}
			\item[(1)] $g[-]$ preserves the class of $m\times n$ SR$(\epsilon)$ matrices.
			\item[(2)] $g[-]$ preserves the class of $d\times d$ SR$(\epsilon)$ matrices.
			\item[(3)]
			\begin{itemize}
				\item[(a)] For $d=1:$ The function $g$ is non-positive.
				
				\item[(b)] For $d=2:$ $g(x)=c\sgn(|x|)$ or $g(x)=c|x|^{\alpha}$ for some $\alpha\in[0,\infty)$ and some $c\leq0$.
				
				\item[(c)] For $d=3:$ $g(x) = \begin{cases}
					c|x|^{\alpha} ~ \text{for some} ~ \alpha\in\{0\}\cup[1,\infty) ~ \text{and some} ~ c\leq0, &\text{if} ~ \epsilon_2\neq\epsilon_3, \\
					c\sgn(|x|) ~ \text{or} ~ c|x|^{\alpha} ~ \text{for some} ~ \alpha\in[0,1] ~ \text{and some} ~ c\leq0, &\text{if} ~ \epsilon_2=\epsilon_3.
				\end{cases}$
				
				\item[(d)] For $d\geq4:$ $g(x) = \begin{cases}
					c|x|^{\alpha} ~ \text{for some} ~ \alpha\in\{0,1\} ~ \text{and some} ~ c\leq0, &\text{if} ~ \epsilon_2\neq\epsilon_3, \\
					c\sgn(|x|) ~ \text{or} ~ c|x|^{\alpha} ~ \text{for some} ~ \alpha\in\{0,1\} ~ \text{and some} ~ c\leq0, &\text{if} ~ \epsilon_2=\epsilon_3.
				\end{cases}$
			\end{itemize}
		\end{itemize}
	\end{cor}
	\begin{proof}
		The equivalence (1)$\iff$(2) is straightforward. We now show (3)$\implies$(2). Let $A$ be any $d\times d$ SR$(\epsilon)$ matrix with $\epsilon_1=-1$, and let $f$ be an entrywise preserver of $d\times d$ SR$(\epsilon^{\prime})$ matrices, where $\epsilon^{\prime}_i=(-1)^{i}\epsilon_i$ for $i=1,\ldots,d$. Define $g(x):=-f(-x)$ for $x<0$. Then, $-f[-A]$ is SR$(\epsilon)$ if and only if $f[-A]$ is SR$(\epsilon^{\prime})$. This observation shows that the entrywise preservers of SR matrices corresponding to the sign pattern $\epsilon$ coincide with those for the sign pattern $\epsilon^{\prime}$. The implication (3)$\implies$(2) of Theorem~\ref{Theorem_SR_Fixed_Classification} now completes the argument.
		
		To show (2)$\implies$(3), suppose $g[-]$ preserves the class of $d\times d$ SR$(\epsilon)$ matrices. The case $d=1$ is trivial. For $d\geq2$, define a function $h:[0,\infty)\to\mathbb{R}$ by $h(x):=-g(-x)$. Then $h$ is an entrywise preserver of SR$(\epsilon^{\prime})$ matrices with $\epsilon_1^{\prime}=1$. By Theorem~\ref{Theorem_SR_Fixed_Classification}, we have the structure of $h$ for each $d\geq2$. Now, by expressing $g$ in terms of $h$ as $g(x)=-h(-x)$ for all $x\leq0$ gives the corresponding form of entrywise preservers for SR$(\epsilon)$ matrices with $\epsilon_1=-1$. This completes the proof.
	\end{proof}
	
	We conclude this section by classifying entrywise preservers of $m\times n$ SSR$(\epsilon)$ matrices with $\epsilon_1=-1$, for all fixed $m,n\geq1$.
	
	\begin{cor}\label{Corollary_SSR_epsilon_1=-1}
		Let $m,n\geq1$ be integers, and let $\epsilon=(\epsilon_1,\ldots,\epsilon_d)$ be a sign pattern with $\epsilon_1=-1$, where $d:=\min\{m,n\}$. Suppose $g:(-\infty,0)\to\mathbb{R}$ is a function. The following are equivalent.
		\begin{itemize}
			\item[(1)] $g[-]$ preserves the class of $m\times n$ SSR$(\epsilon)$ matrices.
			\item[(2)] $g[-]$ preserves the class of $d\times d$ SSR$(\epsilon)$ matrices.
			\item[(3)]
			\begin{itemize}
				\item[(a)] For $d=1:$ $g$ is any negative function.
				
				\item[(b)] For $d=2:$ $g(x)=c|x|^{\alpha}$ for some $\alpha\in(0,\infty)$ and some $c<0$.
				
				\item[(c)] For $d=3:$ $g(x) = \begin{cases}
					c|x|^{\alpha} ~ \text{for some} ~ \alpha\in[1,\infty) ~ \text{and some} ~ c<0, &\text{if} ~ \epsilon_2\neq\epsilon_3, \\
					c|x|^{\alpha} ~ \text{for some} ~ \alpha\in(0,1] ~ \text{and some} ~ c<0, &\text{if} ~ \epsilon_2=\epsilon_3.
				\end{cases}$
				
				\item[(d)] For $d\geq4:$ $g(x)=c|x|$ for some $c<0$.
			\end{itemize}
		\end{itemize}
	\end{cor}
	\begin{proof}
		This follows from Theorem~\ref{Theorem_SSR_Fixed_Classification}, using an argument similar to the proof of Corollary~\ref{Corollary_SR_epsilon_1=-1}.
	\end{proof}
	
	\section{Entrywise preservers of sign regularity for all sign patterns} \label{Section_Not_Fixed_Sign_Pattern}
	
	This section is devoted to characterizing the functions that entrywise preserve the class of all $m\times n$ SR/SSR matrices -- allowing all possible sign patterns -- for each fixed $m,n\geq1$. In this setting, applying such a function to an SR/SSR matrix may result in an SR/SSR matrix with a different sign pattern from that of the original matrix. We begin with the SR case and first establish some fundamental properties of their entrywise preservers.
	
	\begin{prop}\label{Theorem_SR_all_sign_pos_neg}
		Let $m,n\geq1$ be integers with $\max\{m,n\}\geq2$, and let $f:\mathbb{R}\to\mathbb{R}$ be a function that entrywise preserves the class of $m\times n$ SR matrices. Then the following holds:
		\begin{itemize}
			\item[(1)] If $f(0)=0$, then $f|_{(0,\infty)}$ and $f|_{(-\infty,0)}$ are each either a non-positive function or a non-negative function.
			\item[(2)] If $f(0)\neq0$, then $f$ is either non-positive or non-negative on $\mathbb{R}$.
		\end{itemize}
	\end{prop}
	\begin{proof}
		To show $(1)$, suppose $f(0)=0$. We first show that $f$ is either non-negative or non-positive on $(0,\infty)$. If $f\equiv0$ on $(0,\infty)$, then we are done. Let $x_1>0$ such that $f(x_1)\neq0$. Consider the $1\times2$ SR matrix $A = \begin{pmatrix}
			x_1 & x_2
		\end{pmatrix}$, where $x_2>0$. Since $f[A]$ is SR by hypothesis, the signs of $f(x_1)$ and $f(x_2)$ must agree for all $x_2>0$. This implies that $f|_{(0,\infty)}$ is either non-positive or non-negative. The same argument applies to $(-\infty,0)$ by considering negative entries instead.
		
		To prove $(2)$, assume $f(0)\neq0$ and consider the $1\times2$ SR matrix $\begin{pmatrix}
			0 & x
		\end{pmatrix}$, where $x\in\mathbb{R}$. By hypothesis, both $f(0)$ and $f(x)$ have the same sign for all $x\in\mathbb{R}$ with $f(x)\neq0$. Thus, $f$ is either non-positive or non-negative across the entire real line, depending on the sign of $f(0)$. These conclusions also hold for $2\times1$ SR matrices by considering the transpose of the above examples. If $\max\{m,n\}\geq2$, then one of these classes of matrices can be embedded into larger $m\times n$ SR matrices by augmenting them with zeros. This completes the proof.
	\end{proof}
	
	We now proceed to derive further properties of entrywise preservers of sign regularity.
	
	\begin{prop}\label{Theorem_SR_all_sign_monotonic_MMC}
		Let $m,n\geq2$ be integers with $(m,n)\neq(2,2)$. Let $f:\mathbb{R}\to\mathbb{R}$ be an entrywise preserver of the set of all $m\times n$ SR matrices. Then, we have the following:
		\begin{itemize}
			\item[(1)] If $f(0)\neq0$, then $f\equiv f(0)$ on $\mathbb{R}$.
			\item[(2)] The restricted functions $f|_{(-\infty,0]}$ and $f|_{[0,\infty)}$ are monotonic.
			\item[(3)] $f$ satisfies $f(\gamma x)f(\gamma y) = f(\gamma\sqrt{xy})^2$ for all $x,y\geq0$ and for $\gamma\in\{\pm1\}$.
		\end{itemize}
	\end{prop}
	\begin{proof}
		To prove (1), without loss of generality assume $f(0)=c>0$, and suppose $f\not\equiv c$ on $\mathbb{R}$. Then there exists $x_1\in\mathbb{R}\setminus\{0\}$ such that $f(x_1)=a\neq c$, where $a\geq0$ by Proposition~\ref{Theorem_SR_all_sign_pos_neg}(2). Consider the following SR matrix and its image under entrywise application of $f$:
		\begin{align*}
			A_1 = \begin{pmatrix}
				x_1 & 0 & 0 \\
				0 & x_1 & 0
			\end{pmatrix} \implies ~ f[A_1] = \begin{pmatrix}
				a & c & c \\
				c & a & c
			\end{pmatrix}.
		\end{align*}
		Two of the $2\times2$ minors of $f[A_1]$ are $c(a-c)$ and $c(c-a)$, which have opposite signs, contradicting the assumption that $f[A_1]$ is SR. Thus $a=c$, and therefore $f(x)=c$ for all $x\in\mathbb{R}$.
		
		We now prove (2). Using Proposition~\ref{Theorem_SR_all_sign_pos_neg}, without loss of generality assume $f\geq0$ on $[0,\infty)$, and suppose $f$ is not monotonic on this interval. Then there exist $0\leq x_1 < x_2 < x_3 < \infty$ such that $0\leq f(x_1) < f(x_2)$ and $0\leq f(x_3) < f(x_2)$ (without loss of generality). Consider the SR matrix
		\begin{align*}
			A_2 = \begin{pmatrix}
				x_1 & x_2 & x_3 \\
				x_2 & x_2 & x_2
			\end{pmatrix}.
		\end{align*}
		Then $f[A_2]$ is SR. However, its contiguous $2\times2$ minors $f(x_2)(f(x_1)-f(x_2))<0$ and $f(x_2)(f(x_2)-f(x_3))>0$ have opposite signs, a contradiction. By a similar argument one can show that $f$ is monotonic on $(-\infty,0]$.
		
		To prove (3), fix $\gamma\in\{\pm1\}$ and let
		\begin{align*}
			A(x,y,\gamma) := \gamma\begin{pmatrix}
				x & \sqrt{xy} & x \\
				\sqrt{xy} & y & \sqrt{xy}
			\end{pmatrix}, \quad x,y\geq0.
		\end{align*}
		Then $A(x,y,\gamma)$ is SR, so $f[A(x,y,\gamma)]$ is SR. Thus the contiguous $2\times2$ minors of $f[A(x,y,\gamma)]$ gives the identity
		\begin{align*}
			f(\gamma x)f(\gamma y) = f(\gamma\sqrt{xy})^2, \quad \forall ~ x,y\geq0.
		\end{align*}
		
		For $3\times2$ SR matrices, transpose the above matrices. For higher order SR matrices, the result follows by adding zero rows and columns to the above matrices.
	\end{proof}
	
	The above result shows that any function $f$ which entrywise preserves all $m\times n$ SR matrices with $m,n\geq2$ and $(m,n)\neq(2,2)$, must be a non-zero constant function on $\mathbb{R}$ if $f(0)\neq0$. Consequently, we henceforth restrict our attention to functions $f$ satisfying $f(0)=0$. Before moving forward, we summarize the properties of $f$ that have been deduced thus far.
	
	\begin{rem}\label{Remark_all_sign_f1_f2}
		Let $m,n\geq2$ be integers such that $(m,n)\neq(2,2)$, and let $f:\mathbb{R}\to\mathbb{R}$ be a function such that $f[-]$ preserves the set of all $m\times n$ SR matrices. Then
		\begin{itemize}
			\item[(1)] either $f\equiv f(0)$ on $\mathbb{R}$, with $f(0)\neq0$, or
			\item[(2)] $f$ is given by
			\begin{align*}
				f(x) = \begin{cases}
					f_1(x), &x\leq0, \\
					f_2(x), &x\geq0,
				\end{cases}
			\end{align*}
			where $f(0)=0$, and $f_1:(-\infty,0]\to\mathbb{R}$, $f_2:[0,\infty)\to\mathbb{R}$ satisfy the identities
			\begin{align}\label{f_1f_2_MMC}
				f_1(-x)f_1(-y)=f_1(-\sqrt{xy})^2 ~ \text{and} ~ f_2(x)f_2(y)=f_2(\sqrt{xy})^2, \quad \forall ~ x,y\geq0.
			\end{align} 
			Moreover, each of $f_1$ and $f_2$ must satisfy one of the below sign conditions and monotonicity:
			
			\begin{tabular}{p{8cm}p{8cm}}
				\begin{itemize}
					\item[(i)] $f_1\geq0$ and increasing,
					\item[(ii)] $f_1\geq0$ and decreasing,
					\item[(iii)] $f_1\leq0$ and increasing,
					\item[(iv)] $f_1\leq0$ and decreasing,
				\end{itemize} &
				\begin{itemize}
					\item[(v)] $f_2\geq0$ and increasing,
					\item[(vi)] $f_2\geq0$ and decreasing,
					\item[(vii)] $f_2\leq0$ and increasing,
					\item[(viii)] $f_2\leq0$ and decreasing.
				\end{itemize}
			\end{tabular}
		\end{itemize}
		We next study the functions $f_1$ and $f_2$. As a first step, we show the continuity of $f$ on $\mathbb{R}\setminus\{0\}$.
	\end{rem}
	
	\begin{theorem}\label{Theorem_SR_all_sign_continuous}
		Let $m,n\geq2$ be integers such that $(m,n)\neq(2,2)$, and let $f:\mathbb{R}\to\mathbb{R}$ be a function with $f(0)=0$ such that $f[-]$ preserves the set of all $m\times n$ SR matrices. Then $f$ is continuous on $\mathbb{R}\setminus\{0\}$.
	\end{theorem}
	\begin{proof}
		If $f(0)=0$ and $f[-]$ preserves $m\times n$ SR matrices for all $m,n\geq2$ with $(m,n)\neq(2,2)$, then $f$ must be as specified in Remark~\ref{Remark_all_sign_f1_f2}(2). It suffices to show that $f_1|_{(-\infty,0)}$ and $f_2|_{(0,\infty)}$ are continuous.
		
		We first show that $f_1|_{(-\infty,0)}$ and $f_2|_{(0,\infty)}$ are either identically zero or never vanish. Let $\gamma\in\{\pm1\}$. Suppose $f(\gamma a)\neq0$ for some $a>0$; we claim that $f(\gamma x)\neq0$ for all $x>0$. Consider
		\begin{align*}
			A(x,y,\gamma) := \gamma\begin{pmatrix}
				x & xy & x \\
				a & ay & a
			\end{pmatrix}, \quad x,y\geq0.
		\end{align*}
		Then $A(x,y,\gamma)$ is SR, and so is $f[A(x,y,\gamma)]$. This yields the identity
		\begin{align}\label{f_all_sign_not_zero}
			f(\gamma x)f(\gamma ay) = f(\gamma a)f(\gamma xy), \quad \forall ~ x,y\geq0.
		\end{align}
		Now, suppose $f(\gamma x_0)=0$ for some $x_0>0$. Taking $x = \frac{a^2}{x_0}$ and $y=\frac{x_0}{a}$ in~\eqref{f_all_sign_not_zero} gives $f(\gamma a)=0$, a contradiction. Thus, if $f(\gamma a)\neq0$, then $f(\gamma x)\neq0$ for all $x>0$.
		
		If $f_1|_{(-\infty,0)}\equiv0$ or $f_2|_{(0,\infty)}\equiv0$, then the continuity is immediate. We henceforth assume that $f_1|_{(-\infty,0)}$ and $f_2|_{(0,\infty)}$ are nowhere zero.
		
		We first show $f_2|_{(0,\infty)}$ is continuous. If $f_2$ satisfies condition~(v), then continuity follows from Theorem~\ref{Theorem_SR_Fixed_1_continuous}. If $f_2$ satisfies~(viii), then define the auxiliary function
		\begin{align}\label{auxiliary_g2}
			g_2:(0,\infty)\to\mathbb{R} ~ \text{by} ~ g_2(x):= -f_2(x).
		\end{align}
		Now, $g_2$ is positive and monotonic increasing on $(0,\infty)$, and $g_2(x)g_2(y)=g_2(\sqrt{xy})^2$ for all $x,y>0$. By Theorem~\ref{Theorem_SR_Fixed_1_continuous}, $g_2$ is continuous and hence $f_2$ is also continuous.
		
		Assume $f_2$ satisfies~(vi). We establish continuity of $f_2$ on an arbitrary open interval $I^+=(0,\rho)$, where $0<\rho\leq\infty$. Fix $t\in I^+$, and choose $0<\delta\leq\min\left\{\frac{t}{5}, \frac{\rho-t}{4}\right\}$. Then
		\begin{align}\label{All_sign_continuous_interval}
			0<t+\delta\leq \sqrt{(t+4\delta)(t-\delta)}<\rho.
		\end{align}
		Since $f_2$ is non-negative and decreasing on $[0,\infty)$, using \eqref{f_1f_2_MMC} and \eqref{All_sign_continuous_interval}, we have
		\begin{align*}
			f_2(t+\delta)\geq f_2\left(\sqrt{(t+4\delta)(t-\delta)}\right) = \sqrt{f_2(t+4\delta)f_2(t-\delta)}.
		\end{align*}
		Letting $\delta\to0^+$ and using that $f_2>0$ on $I^+$, we conclude
		\begin{align*}
			f_2(t^+) \geq \sqrt{f_2(t^+)f_2(t^-)} > 0 \implies f_2(t^+)\geq f_2(t^-).
		\end{align*}
		Since $f_2$ is decreasing, we also have $f_2(t^+)\leq f_2(t)\leq f_2(t^-)$, so $f(t^-)=f(t^+)$, proving continuity for an arbitrary $t\in I^+$. Thus $f_2|_{(0,\infty)}$ is continuous. 
		
		To show that $f_2$ satisfying condition~(vii) is continuous, consider the function $g_2$ as defined in~\eqref{auxiliary_g2}. Since $g_2$ satisfies case~(vi), it follows that $g_2$ is continuous, and therefore $f_2|_{(0,\infty)}$ is also continuous.
		
		For $f_1$, define $g_1:(0,\infty)\to\mathbb{R}$ by $g_1(x):=f_1(-x)$. Then $g_1$ satisfies the same functional identity given in the above remark and fits into one of the cases~(v)~to~(viii), so it is continuous on $(0,\infty)$. Therefore, $f_1|_{(-\infty,0)}$ is continuous. This completes the proof.
	\end{proof}
	
	The following result partially determines the functions that preserve the set of all $m\times n$ SR matrices under entrywise application, for all fixed integers $m, n\geq2$ with $(m,n)\neq(2,2)$.
	
	\begin{theorem}\label{Theorem_SR_All_Sign_Cont_Power_Signum}
		Let $m,n\geq2$ be integers such that $(m,n)\neq(2,2)$. Let $f:\mathbb{R}\to\mathbb{R}$ be a function with $f(0)=0$ such that $f[-]$ preserves the set of all $m\times n$ SR matrices. Then
		\begin{align*}
			f(x) = \begin{cases}
				c_1 |x|^{\alpha_1} ~ \text{for some} ~ \alpha_1\geq0 ~ \text{and some} ~ c_1\in\mathbb{R}, & \text{if } x<0, \\
				c_2 x^{\alpha_2} ~ \text{for some} ~ \alpha_2\geq0 ~ \text{and some} ~ c_2\in\mathbb{R}, & \text{if } x>0.
			\end{cases}
		\end{align*}
	\end{theorem}
	\begin{proof}
		By Theorem~\ref{Theorem_SR_all_sign_continuous}, $f$ is continuous on $\mathbb{R}\setminus\{0\}$. Define $f_1:=f|_{(-\infty,0]}$ and $f_2:=f|_{[0,\infty)}$. Taking $a=\gamma=1$ in~\eqref{f_all_sign_not_zero}, we obtain
		\begin{equation}\label{SR_all_sign_f_relation}
			f_2(x)f_2(y)=f_2(1)f_2(xy), \quad \forall ~ x,y\geq0.
		\end{equation}
		By Proposition~\ref{Theorem_SR_all_sign_pos_neg}(1), there exists $\gamma^{\prime}\in\{\pm1\}$ such that $\gamma^{\prime} f_2\geq0$ on $[0,\infty)$. We now consider two cases.
		\begin{itemize}[leftmargin=0pt,label={}]
			\item \textbf{Case 1.} $\gamma^{\prime} f_2(1)=0$: Setting $x=y\geq0$ in \eqref{SR_all_sign_f_relation} yields $f_2\equiv0$ on $[0,\infty)$.
			
			\item \textbf{Case 2.} $\gamma^{\prime} f_2(1)>0$: Similar to the first half of the proof of Theorem~\ref{Theorem_SR_all_sign_continuous}, one can show that $\gamma^{\prime} f_2(x)>0$ for all $x>0$. Now define
			\begin{align}\label{Function1_g_h}
				g(x) := \frac{f_2(x)}{f_2(1)}, ~ x>0, \qquad h(y):=\log g(e^y), ~ y\in\mathbb{R.}
			\end{align}
			Then $g$ is positive, continuous, and multiplicative on $(0,\infty)$. Thus, $h$ is continuous on $\mathbb{R}$ and satisfies the Cauchy functional equation. By Theorem~\ref{Theorem_Darboux}, $h(y)=h(1)y$ for all $y\in\mathbb{R}$. Translating back, we obtain
			\begin{align}\label{f_2_all_sign}
				f_2(x)=f_2(1)x^{h(1)}, \quad \forall~x>0,
			\end{align}
			with $f_2(1)\in\mathbb{R}\setminus\{0\}$. We now show $h(1)\geq0$. Suppose instead that $h(1)<0$. Let $0<x_1<x_2<\infty$, then $0=\gamma^{\prime} f_2(0) <\gamma^{\prime} f_2(x_1)>\gamma^{\prime} f_2(x_2)>0$. Consider the matrix
			\begin{align*}
				A = \begin{pmatrix}
					0 & x_1 & x_2 \\
					x_1 & x_1 & x_1
				\end{pmatrix}\oplus0_{(m-2)\times(n-3)},
			\end{align*}
			which is SR, but $f_2[A]$ fails to be SR, as its contiguous $2\times2$ minors have opposite signs, a contradiction.
		\end{itemize}
		Combining cases 1 and 2, we conclude that $f_2(x)=c_2x^{\alpha}$ for all $x>0$, with $\alpha:=h(1)\geq0$ and $c_2:=f_2(1)\in\mathbb{R}$. 
		
		To determine the form of $f_1$, define the auxiliary function $k:[0,\infty)\to\mathbb{R}$ by $k(x):=f_1(-x)$. Now, $k|_{(0,\infty)}$ is continuous and $k(0)=0$. By a similar argument as above, $k(x)=k(1)x^{q(1)}$ for all $x>0$, where the function $q$ is defined similarly to $h$ in~\eqref{Function1_g_h}. Therefore, $f_1(x)=f_1(-1)|x|^{q(1)}$ for all $x<0$ with $q(1)\geq0$ and $f_1(-1)\in\mathbb{R}$.
		%
	\end{proof}
	
	Observe that in the above theorem, if $c_1=c_2=0$, then $f\equiv0$ on $\mathbb{R}$. In contrast, when $c_1,c_2\neq0$, the function $f$ depends on the value of $\alpha$: for $\alpha=0$, $f$ corresponds to a (scaled) signum function on the intervals $(-\infty,0]$ and $[0,\infty)$; whereas for $\alpha>0$, $f$ represents a power function on each of these intervals. Thus, our goal reduces to study the following questions for each fixed dimension $m,n\geq1$:
	\begin{itemize}
		\item[(i)] Does the entrywise application of the signum function preserve the set of $m\times n$ SR matrices?
		\item[(ii)] Which power functions with positive exponents, when applied entrywise, preserve the class of $m\times n$ SR matrices?
	\end{itemize}
	We address the first question in the next theorem.
	
	\begin{theorem}\label{Theorem_SR_AllSign_Signum}
		Let $m,n\geq1$ be integers. Consider the function $f:\mathbb{R}\to\mathbb{R}$ given by $f(x)=c\sgn|x|$, where $c\in\mathbb{R}\setminus\{0\}$. Then
		\begin{itemize}
			\item[(1)] $f[-]$ preserves the set of all $n\times n$ SR matrices if and only if $n\leq3$.
			\item[(2)] If $m\neq n$, then $f[-]$ preserves the set of all $m\times n$ SR matrices if and only if $\min\{m,n\}\leq2$.
		\end{itemize}
	\end{theorem}
	\begin{proof}
		We begin by proving (2). Let $m\neq n$ and $d:=\min\{m,n\}$. If $d\leq2$, the conclusion holds immediately. To prove the converse, we need to construct examples of $m\times n$ SR matrices with $d\geq3$ whose images under $f[-]$ are not SR. First, take $d=3$. Consider the following matrix and its image under the entrywise map $f$.
		\begin{align*}
			A = \begin{pmatrix}
				6 & 1 & 0 & 0 \\
				1 & 1 & 1 & 0 \\
				0 & 4 & 5 & 1
			\end{pmatrix}\implies f[A] = \begin{pmatrix}
				c & c & 0 & 0 \\
				c & c & c & 0 \\
				0 & c & c & c
			\end{pmatrix}.
		\end{align*}
		All minors of $A$ are non-negative, so $A$ is SR. However, in $f[A]$, the two contiguous $3\times3$ submatrices have determinants of opposite signs. For $d\geq4$, the argument holds by padding $A$ with zero rows or columns to get the desired dimensions.
		
		We now show (1). Note that for each $n\leq3$, $f[-]$ preserves all $n\times n$ SR matrices. For $n\geq4$, we can extend the above example to a square matrix by adding zero rows or columns. The resulting matrix remains SR, but its image under $f[-]$ fails to be SR. This completes the proof.
	\end{proof}
	
	We now answer the second question posed above: classify the power functions with positive exponents that entrywise preserve the class of $m\times n$ SR matrices, for fixed integers $m,n\geq1$.
	
	\begin{theorem}\label{Theorem_SR_AllSign_Power}
		Let $m,n\geq1$ be integers, and define $d:=\min\{m,n\}$. Consider the function $f:\mathbb{R}\to\mathbb{R}$ given by $f(x)=|x|^{\alpha}$, with $\alpha>0$. Then the following statements are equivalent.
		\begin{itemize}
			\item[(1)] $f[-]$ preserves the set of all $m\times n$ SR matrices.
			\item[(2)] We have the following two cases:
			\begin{itemize}
				\begin{tabular}{p{8cm}p{8cm}}
					\item[(a)] If $m=n$, then:
					\begin{itemize}
						\item[(i)] For $n\leq3:$ $\alpha\in(0,\infty)$.
						\item[(ii)] For $n\geq4:$ $\alpha=1$.
					\end{itemize} &
					\item[(b)] If $m\neq n$, then:
					\begin{itemize}
						\item[(i)] For $d\leq2:$ $\alpha\in(0,\infty)$.
						\item[(ii)] For $d\geq3:$ $\alpha=1$.
					\end{itemize}
				\end{tabular}
			\end{itemize}
		\end{itemize}
	\end{theorem}
	\begin{proof}
		We begin by proving (2)$\implies$(1). First, consider the case $m=n=3$. Let $A$ be a $3\times 3$ SR matrix. For all $\alpha>0$, the $2\times2$ minors of $f[A]$ have the same sign. Since $f[A]$ has a unique $3\times 3$ minor, $\det f[A]$, it follows that $f[A]$ is sign regular for all $\alpha>0$. The remaining cases are straightforward. We next show (1)$\implies$(2)(a). The case $n\leq3$ is trivial. For $n=4$, we show that for each $\alpha\in(0,\infty)\setminus\{1\}$, there exists a $4\times4$ SR matrix $A$ such that $f[A]$ contains at least two $3\times3$ minors with opposite signs. For $n\geq5$, the matrix $A\oplus0_{(n-4)\times(n-4)}$ provides a suitable example.
		
		Let $n=4$. To exclude $\alpha\in(0,1)$, consider the following matrix along with two of its submatrices:
		\begin{align*}
			A=\begin{pmatrix}
				3 & 1 & 2 & 1 \\
				1 & 1 & 4 & 3 \\
				1 & 2 & 9 & 8 \\
				0 & 0 & 0 & 0
			\end{pmatrix}, \qquad A_1 = \begin{pmatrix}
				3 & 1 & 2 \\
				1 & 1 & 4 \\
				1 & 2 & 9
			\end{pmatrix}, \qquad A_2 = \begin{pmatrix}
				1 & 2 & 1 \\
				1 & 4 & 3 \\
				2 & 9 & 8
			\end{pmatrix}.
		\end{align*}
		By a simple calculation, one can verify that $A$ is an SR matrix with sign pattern $\epsilon=(1,1,1,\epsilon_4)$, where $\epsilon_4\in\{\pm1\}$. We claim that for all $\alpha\in(0,1)$, we have $\det A_1^{\circ\alpha}<0$ and $\det A_2^{\circ\alpha}>0$.
		
		In the proof of Theorem~\ref{Theorem_SR_power_classification}, we established that $\det A_1^{\circ\alpha}<0$ for all $\alpha\in(0,1)$. We next show that $\det A_2^{\circ\alpha}>0$ for all $\alpha\in(0,1)$. Consider the expression
		\[\det A_2^{\circ\alpha} = -8^{\alpha} + 9^{\alpha} + 12^{\alpha} -16^{\alpha} -27^{\alpha} + 32^{\alpha}.\]
		It follows from Theorem~\ref{Jameson} that the above function has at most three real roots. Furthermore, a Taylor expansion of $\det A_2^{\circ\alpha}$ at $\alpha=0$ gives $\det A_2^{\circ\alpha} = c\alpha^2 + O(\alpha^3)$, where
		\begin{align*}
			c = \frac{1}{2}\left(-(\log8)^2 +(\log9)^2 +(\log12)^2 -(\log16)^2 -(\log27)^2 +(\log32)^2\right)>0.
		\end{align*}	
		This shows that $\det A_2^{\circ\alpha}$ has a double root at $\alpha=0$, and moreover
		\[\displaystyle{\lim_{\alpha\to0^+}\frac{\det A_2^{\circ\alpha}}{\alpha^2}} = c>0.\] 
		Thus $\det A_2^{\circ\alpha}>0$ for sufficiently small $\alpha>0$. Also note that $\det A_2^{\circ\alpha}>0$ at $\alpha=1$. Thus, $\det A_2^{\circ\alpha}>0$ for all $\alpha\in(0,1)$; otherwise, it would have more than three real roots, contradicting Theorem~\ref{Jameson}. Hence, $A^{\circ\alpha}$ is not SR for all $\alpha\in(0,1)$.
		
		To discard powers in $(1,\infty)$, examine the following matrix and its submatrices for $0\leq t<1$:
		\begin{align*}
			A(t) :=\begin{pmatrix}
				1 & 1 & 1 & 1 \\
				1-\frac{1}{5}t & 1 & 1+2t & 1+3t \\[2pt]
				1-\frac{1}{2}t & 1 & 1+5t & 1+7t \\[2pt]
				0 & 0 & 0 & 0
			\end{pmatrix}, ~ A_1(t) :=\begin{pmatrix}
				1 & 1 & 1 \\
				1-\frac{1}{5}t & 1 & 1+2t \\[2pt]
				1-\frac{1}{2}t & 1 & 1+5t
			\end{pmatrix}, ~ A_2(t) :=\begin{pmatrix}
				1 & 1 & 1 \\
				1 & 1+2t & 1+3t \\
				1 & 1+5t & 1+7t
			\end{pmatrix}.
		\end{align*}
		For the matrix $A(t)$, observe that all $2\times2$ minors take the form $ut+vt^2$ (with $u,v\geq0$) or $t(1-t)$; all $3\times3$ minors are of the form $-ut^2$ for some $u\geq0$; and $\det A(t)=0$. Thus, $A(t)$ is an SR matrix for all $0<t<1$ with sign pattern $\epsilon=(1,1,-1,\epsilon_4)$, where $\epsilon_4\in\{\pm1\}$. We now show that for every $\alpha>1$, there exists $t\in(0,1)$ such that $\det A_1(t)^{\circ\alpha}>0>\det A_2(t)^{\circ\alpha}$. 
		
		Let $\alpha_0\in(1,\infty)$. Expanding $\det A_1(t)^{\circ\alpha_0}$ using a Taylor series at $t=0$ gives
		\begin{align*}
			\det A_1(t)^{\circ\alpha_0} = \frac{33}{20}(\alpha_0^3-\alpha_0^2)t^3 + O(t^4) \implies 	\displaystyle{\lim_{t\to0^+}\frac{\det A_1(t)^{\circ\alpha_0}}{t^3}} = \frac{33}{20}(\alpha_0^3-\alpha_0^2)>0.
		\end{align*}
		Thus, for sufficiently small $t>0$, say $t_1$, we have $\det A_1(t_1)^{\circ\alpha_0}>0$. Similarly, a Taylor expansion of $\det A_2(t)^{\circ\alpha_0}$ at $t=0$ yields
		\begin{align*}
			\det A_2(t)^{\circ\alpha_0} = -\alpha_0^2t^2 + O(t^3) \implies \displaystyle{\lim_{t\to0^+}\frac{\det A_2(t)^{\circ\alpha_0}}{t^2}} = -\alpha_0^2<0.
		\end{align*}
		Hence, for sufficiently small $t>0$, say $t_2$, we have $\det A_2(t_2)^{\circ\alpha_0}<0$. Letting $\tilde{t}:=\min\{t_1,t_2\}$, we conclude that the matrix $A(\tilde{t})^{\circ\alpha_0}$ contains two $3\times3$ minors of opposite signs, thereby establishing (1)$\implies$(2)(a).
		
		We now show (1)$\implies$(2)(b). Note that the case $d\leq2$ is trivial. For $d=3$, the required examples are provided by the $3\times4$ submatrices of $A$ and $A(t)$, obtained by removing the final row. For $d\geq4$, we embed these matrices into larger dimensions SR matrices via zero-padding or transposition. This concludes the proof.
	\end{proof}
	
	With these results in hand, we now prove Theorem~\ref{Theorem_SR_Non-Fixed_Classification}.
	
	\begin{proof}[Proof of Theorem \ref{Theorem_SR_Non-Fixed_Classification}]
		The implications (2)(a)$\implies$(1) and (2)(b)$\implies$(1) are immediate, so we now focus on proving the converse: (1)$\implies$(2). The cases $m=n=1,2$, and $m\neq n$ with $d=1$ are straightforward. We thus consider the remaining cases: $m=n\geq3$ and $m\neq n$ with $d\geq 2$.
		
		If $f(0)\neq0$, by Proposition~\ref{Theorem_SR_all_sign_monotonic_MMC}(1), $f(x)=f(0)$ for all $x\in\mathbb{R}$. If $f(0)=0$, then by Theorem~\ref{Theorem_SR_All_Sign_Cont_Power_Signum}, the functions $f|_{(-\infty,0]}$ and $f|_{[0,\infty)}$ are each either the zero function, a power function, or a scaled signum function. By Theorems~\ref{Theorem_SR_AllSign_Signum}~and~\ref{Theorem_SR_AllSign_Power}, one has a complete classification of the functions among these which preserve SR matrices entrywise. This completes the proof.
	\end{proof}
	
	We now address the problem of characterizing all functions that preserve the class of $m\times n$ SSR matrices under entrywise application, for all fixed dimensions $m$ and $n$. As a first step, we derive several key properties that any such function must satisfy. We begin with the following observation.
	
	\begin{rem}\label{Remark_SSR_Property_Inherit}
		The properties of entrywise preservers of $m\times n$ SSR matrices are inherited by preservers of $m^{\prime}\times n^{\prime}$ SSR matrices for all $m^{\prime}\geq m$ and $n^{\prime}\geq n$. 
\end{rem}

We use this fact in the subsequent proofs.

\begin{theorem}\label{Theorem_SSR_all_sign_pos_neg}
	Let $m,n\geq1$ be integers such that $\min\{m,n\}=1$, and let $f:\mathbb{R}\setminus\{0\}\to\mathbb{R}$. The following are equivalent.
	\begin{itemize}
		\item[(1)] $f[-]$ preserves the set of all $m\times n$ SSR matrices.
		\item[(2)]
		\begin{itemize}
			\item[(a)] If $m=n=1$, then $f(x)\neq0$ for all $x\in\mathbb{R}\setminus\{0\}$.
			\item[(b)] If $m\neq n$, then $f|_{(-\infty,0)}$ and $f|_{(0,\infty)}$ are each either a positive or a negative function.
		\end{itemize}
	\end{itemize}
\end{theorem}
\begin{proof}
	Note that (2)$\implies$(1) and (1)$\implies$(2)(a) are immediate. We next show (1)$\implies$(2)(b). Using Remark~\ref{Remark_SSR_Property_Inherit}, the case (1)$\implies$(2)(a) implies that $f$ is nowhere zero. Now assume that $f$ is neither positive nor negative on $(0,\infty)$. Then there exist $0<x_1<x_2<\infty$ such that $f(x_1)f(x_2)<0$. Note that the $1\times2$ matrix $A_1 = \begin{pmatrix}
		x_1 & x_2
	\end{pmatrix}$ is SSR, but since the signs of $f(x_1)$ and $f(x_2)$ differ, the matrix $f[A_1]$ is not SSR, violating the assumption. Therefore, $f$ must be either positive or negative on $(0,\infty)$. A similar argument applies on $(-\infty,0)$. For $2\times1$ SSR matrices, the same conclusion follows by transposing $A_1$. When $\min\{m,n\}=1$ and $\max\{m,n\}\geq3$, the result follows by Remark~\ref{Remark_SSR_Property_Inherit}, completing the proof.
\end{proof}
The preceding result provides a complete characterization of entrywise preservers of $m\times n$ SSR matrices for all $m,n\geq1$ with $\min\{m,n\}$=1. We now consider the case where $m,n\geq2$.

\begin{prop}\label{Proposition_geq2_SSR}
	Let $m,n\geq2$ be integers and suppose $f:\mathbb{R}\setminus\{0\}\to\mathbb{R}$ entrywise preserves the set of all $m\times n$ SSR matrices. Then the functions $f|_{(-\infty,0)}$ and $f|_{(0,\infty)}$ are injective, and either positive or negative.
\end{prop}
\begin{proof}
	By Remark~\ref{Remark_SSR_Property_Inherit}, Theorem~\ref{Theorem_SSR_all_sign_pos_neg}(2)(b) implies that the restrictions $f|_{(-\infty,0)}$ and $f|_{(0,\infty)}$ are each either positive or negative. Next, to show that $f|_{(0,\infty)}$ is injective, suppose, for contradiction, that there exist $0<x_1\neq x_2<\infty$ such that $f(x_1)=f(x_2)$. Then $A=\begin{pmatrix}
		x_1 & x_2 \\
		1 & 1
	\end{pmatrix}$ is SSR, but $\det f[A]=0$, a contradiction. A similar argument applies for the interval $(-\infty,0)$.
\end{proof}

We now classify a class of functions that entrywise preserve the set of all $2\times2$ SSR matrices with additional constraints.

\begin{lemma}\label{Lemma_2by2_SSR}
	Let $f:\mathbb{R}\setminus\{0\}\to\mathbb{R}$ be a function that satisfies the following conditions:
	\begin{itemize}
		\item[(1)] $f|_{(-\infty,0)}$ and $f|_{(0,\infty)}$ are each either a positive function or a negative function.
		\item[(2)] The restrictions $f|_{(-\infty,0)}$ and $f|_{(0,\infty)}$ are both injective functions.
		\item[(3)] The functions $g_1(x)=\frac{f(x)}{f(-1)}$ and $g_2(x)=\frac{f(x)}{f(1)}$ are multiplicative on $(-\infty,0)$ and $(0,\infty)$, respectively.
	\end{itemize}
	Then $f[-]$ preserves the set of all $2\times2$ SSR matrices. The converse holds if the restricted functions $f|_{(-\infty,0)}$ and $f|_{(0,\infty)}$ are onto on their respective codomains (see part~(1)).
\end{lemma}
\begin{proof}
	Let $A=\begin{pmatrix}
		x_1 & x_2 \\
		x_3 & x_4
	\end{pmatrix}$ be a $2\times2$ SSR matrix. Without loss of generality, assume that $x_i>0$. By condition~(1), the entries of $f[A]$ are non-zero and have the same sign. Next, we claim that $\det f[A]\neq0$. Note that
	\begin{align*}
		\det f[A] = f(1)^2\left(g_2(x_1)g_2(x_4)- g_2(x_2) g_2(x_3)\right)=f(1)^2\left(g_2(x_1x_4)-g_2(x_2x_3)\right)\neq0,
	\end{align*}
	where the second equality follows from condition~(3), and the final inequality holds because $x_1x_4\neq x_2x_3$, and $f|_{(0,\infty)}$ is non-zero and injective; consequently, the same holds for $g_2$.
	
	We now show the converse assuming that the restricted functions $f|_{(-\infty,0)}$ and $f|_{(0,\infty)}$ are onto on their respective codomains. Suppose $f[-]$ preserves the set of all $2\times2$ SSR matrices. Then, by Proposition~\ref{Proposition_geq2_SSR}, $f$ satisfies conditions~(1)~and~(2). It remains to show that $g_1$ and $g_2$ are multiplicative. To show $g_2$ is multiplicative, it suffices to show that $f(xy)f(1)=f(x)f(y)$ for all $x,y>0$. Let $x,y>0$. Then, there exists $\alpha$ in the codomain of $f|_{(0,\infty)}$ such that $\alpha f(xy)=f(x)f(y)$. Since $f|_{(0,\infty)}$ is onto, there exists $\beta>0$ such that $f(\beta)=\alpha$. Take $A=\begin{pmatrix}
		x & xy \\
		\beta & y
	\end{pmatrix}$. Then $\det f[A]=0$. Since $f[-]$ preserves $2\times2$ SSR matrices, it follows that $\det A=0$. Thus $\beta=1$, and hence $g_2$ is multiplicative. Similarly, one can prove that $g_1$ is multiplicative. This concludes the proof.
\end{proof}

We next establish that entrywise preservers of $m\times n$ SSR matrices, with $m,n\geq2$ and $(m,n)\neq(2,2)$, are continuous on $\mathbb{R}\setminus\{0\}$. Furthermore, we show that such functions must be power functions on each of the intervals $(-\infty,0)$ and $(0,\infty)$.

\begin{theorem}\label{Theorem_SSR_all_sign_f_form_power}
	Let $m,n\geq2$ be integers with $(m,n)\neq(2,2)$. Suppose $f:\mathbb{R}\setminus\{0\}\to\mathbb{R}$ is a function that entrywise preserves the set of all $m\times n$ SSR matrices. Then
	\begin{align*}
		f(x) = \begin{cases}
			c_1|x|^{\alpha} ~ \text{for some $\alpha,c_1\in\mathbb{R}\setminus\{0\}$}, & \text{if} ~ x<0, \\
			c_2x^{\alpha} ~ \text{for some $\alpha,c_2\in\mathbb{R}\setminus\{0\}$}, & \text{if} ~ x>0.
		\end{cases}
	\end{align*}
\end{theorem}
\begin{proof}
	By Proposition~\ref{Proposition_geq2_SSR}, the restrictions $f|_{(-\infty,0)}$ and $f|_{(0,\infty)}$ must be either positive or negative, and each restriction must be injective. By adapting the argument in the proof of Proposition~\ref{Theorem_SR_all_sign_monotonic_MMC}(2), one shows that $f|_{(-\infty,0)}$ and $f|_{(0,\infty)}$ are strictly monotonic. Now, define the set
	\[\mathcal{D} := \{x>0 : f ~ \text{is discontinuous at} ~ x\}.\]
	Since strictly monotonic functions can have at most countably many discontinuities, the set $\mathcal{D}$ is countable. Therefore, there exists some $a>0$ at which $f$ is continuous. Consider the SSR matrix
	\begin{align*}
		A(x,y,\epsilon):= \begin{pmatrix}
			ax & axy & ax  \\
			a-\epsilon & ay & a+\epsilon
		\end{pmatrix}, \quad x,y>0 ~ \text{and} ~ a>\epsilon>0.
	\end{align*}
	By hypothesis, $f[A(x,y,\epsilon)]$ is SSR. Examining the contiguous $2\times2$ minors of this matrix and taking the limit as $\epsilon\to0^+$, we obtain
	\begin{align}\label{f_identity_SSR_all_sign}
		f(ax)f(ay) = f(a)f(axy), \quad \forall~x,y>0.
	\end{align} 
	Define the functions
	\begin{align}\label{Function2_g_h}
		g(x) := \frac{f(ax)}{f(a)}, ~ x>0, \qquad h(y):=\log g(e^y), ~ y\in\mathbb{R}.
	\end{align}
	Then $g$ is well-defined, positive, multiplicative, and has at most countably many discontinuities (viz. $x>0$ such that $x\mapsto ax\in \mathcal{D}$). It follows that $h$ is well-defined and also has at most countably many discontinuities (viz. $y\in\mathbb{R}$ such that $y\mapsto ae^y\in \mathcal{D}$). Since $h$ satisfies the Cauchy functional equation, by Theorem~\ref{Theorem_Darboux}, $h(y)=h(1)y$ for all $y\in\mathbb{R}$. This implies
	\begin{align*}
		f(ax) = f(a)x^{h(1)}, ~ \forall ~ x>0, \quad \text{or} \quad f(x) = \frac{f(a)}{a^{h(1)}}x^{h(1)}, ~ \forall ~ x>0.
	\end{align*}
	Since $f|_{(0,\infty)}$ is either positive or negative, and injective, it follows that $f(a)\neq0$ and that $h(1)\neq0$, respectively.
	
	For the interval $(-\infty,0)$, define $p:(0,\infty)\to\mathbb{R}$ by $p(x):=f(-x)$. Since $p$ preserves SSR matrices with positive entries, it follows that $p(x)=\frac{p(\tilde{a})}{\tilde{a}^{q(1)}}x^{q(1)}$ for all $x>0$, where the function $q$ is defined similarly to $h$ in~\eqref{Function2_g_h} and $p$ is continuous at $\tilde{a}>0$. Consequently, $f(x) = \frac{f(-\tilde{a})}{\tilde{a}^{q(1)}}(-x)^{q(1)}$ for all $x<0$ with $f(-\tilde{a}),q(1)\neq0$.
	
	For the $3\times2$ case, the result follows by transposing the above $2\times3$ matrix. For higher-dimensional SSR matrices, the conclusion follows by appending rows and columns to the above matrix (or its transpose) using Theorem~\ref{Theorem_line_insertion_SSR}. This completes the proof.
\end{proof}

The preceding theorem shows that if $f[-]$ preserves the set of all $m\times n$ SSR matrices for fixed integers $m,n\geq2$ with $(m,n)\neq(2,2)$, then $f$ must be a power function on the intervals $(-\infty,0)$ and $(0,\infty)$. The following theorem provides a classification of the admissible exponents for these power functions.

\begin{theorem}\label{Theorem_SSR_all_sign_power_classification}
	Let $m,n\geq1$ be integers, and define $d:=\min\{m,n\}$. The following statements are equivalent for $\alpha\in\mathbb{R}$.
	\begin{itemize}
		\item[(1)] The map $x\mapsto |x|^{\alpha}$ entrywise preserves the set of all $m\times n$ SSR matrices.
		\item[(2)]
		\begin{itemize}
			\item[(a)] For $d=1:$ $\alpha\in\mathbb{R}$.
			\item[(b)] For $d=2:$ $\alpha\in\mathbb{R}\setminus\{0\}$.
			\item[(c)] For $d\geq3:$ $\alpha=1$.
		\end{itemize}
	\end{itemize}
\end{theorem}
\begin{proof}
	The implication (2)$\implies$(1) is straightforward. We now show (1)$\implies$(2). The cases where $d=1$ and $d=2$ are trivial, so we focus on $d\geq3$.
	
	Let $d=3$. For each $\alpha\in\mathbb{R}\setminus\{1\}$, we construct a $3\times3$ SSR matrix $A$ such that $\det A^{\circ\alpha}=0$, thereby showing that such powers do not preserve $3\times3$ SSR matrices. We first present such examples for each $\alpha<0$. Consider the following matrix
	\begin{align*}
		A_1(t) := \begin{pmatrix}
			1 & 1 & 1 \\
			1 & 1+t & 1+2t \\
			1 & 1+3t & 1+7t
		\end{pmatrix}, \quad t\geq0.
	\end{align*}
	For this matrix, note that each $2\times2$ minor is of the form $ut+vt^2$, with $u>0, v\geq0$; and $\det A_1(t) = t^2$. Thus, $A_1(t)$ is an SSR matrix for all $t>0$. Fix $\alpha_0<0$, and define
	\[F_1(t):=\det A_1(t)^{\circ\alpha_0}, \quad \text{for} ~  t\geq0.\] 
	We claim that $F_1(t)=0$ for some $t>0$. A Taylor expansion of $F_1$ at $t=0$ yields
	\begin{align*}
		F_1(t) = \alpha_0^2t^2 + O(t^3) ~ \implies ~ \displaystyle{\lim_{t\to0^+}\frac{F_1(t)}{t^2}} = \alpha_0^2.
	\end{align*}
	Thus, for sufficiently small $t>0$, say $t_1$, we have $F_1(t_1)>0$. Next, consider
	\[\det A_1(1)^{\circ\alpha} = -2^{\alpha} + 3^{\alpha} + 4^{\alpha} -8^{\alpha} -12^{\alpha} + 16^{\alpha}.\] 
	By Theorem~\ref{Jameson}, this expression has at most three real roots. Its Taylor expansion at $\alpha=0$ gives $\det A_1(1)^{\circ\alpha} = c\alpha^2 + O(\alpha^3)$, where
	\begin{align*}
		c = \frac{1}{2}(-(\log2)^2 +(\log3)^2 +(\log4)^2 -(\log8)^2 -(\log12)^2 +(\log16)^2) < 0.
	\end{align*}
	Thus, $\det A_1(1)^{\circ\alpha}$ has a double root at $\alpha=0$ and
	\begin{align*}
		\displaystyle{\lim_{\alpha\to0^+}\frac{\det A_1(1)^{\circ\alpha}}{\alpha^2}} = c <0.
	\end{align*}
	Hence, $\det A_1(1)^{\circ\alpha}<0$ for sufficiently small $\alpha>0$. Since $\det A_1(1)^{\circ\alpha}>0$ at $\alpha=1$, it follows that $\det A_1(1)^{\circ\alpha}$ must have a root in the interval $(0,1)$. Thus $\det A_1(1)^{\circ\alpha}<0$ for all $\alpha<0$, including $\alpha_0$. In particular, $F_1(1)<0$. By continuity of $F_1$, the intermediate value theorem guarantees the existence of $\tilde{t}\in(t_1,1)$ such that $F_1(\tilde{t})=0$.
	
	To discard powers in the interval $(0,1)$, consider the matrix
	\begin{align*}
		A_2(t,\delta) := \begin{pmatrix}
			1 & 1 & 1 \\[2pt]
			1 & 1+\frac{1}{3}t & 1+\frac{2}{3}t \\[2pt]
			1 & 1+t & 1+(2+\delta)t 
		\end{pmatrix}, \quad t,\delta \geq 0.
	\end{align*}
	Note that all $2\times2$ minors take the form $ut + v\delta t + w\delta t^2$, where $u>0$, and $v,w \geq0$; and $\det A_2(t,\delta) = \frac{1}{3}\delta t^2$. Consequently, $A_2(t,\delta)$ is an SSR matrix for all positive values of $t$ and $\delta$. Fix $\alpha_0\in(0,1)$, and define the function
	\begin{align*}
		F_2(t, \delta):=\det A_2(t,\delta)^{\circ\alpha_0}, \quad \text{for} ~  t,\delta \geq 0.
	\end{align*}
	We aim to show the existence of $t,\delta>0$ such that $F_2(t,\delta)=0$. To establish this, we proceed through the following steps:
	\begin{itemize}[leftmargin=0pt,label={}]
		\item \textbf{Step 1.} Choose $t_1>0$, and identify a corresponding $\delta_1>0$ such that $F_2(t_1,\delta_1)<0$.
		
		\item \textbf{Step 2.} With $\delta_1$ fixed, find $t_2>0$ for which $F_2(t_2,\delta_1)>0$.
		
		\item \textbf{Step 3.} By continuity of $F_2(t,\delta_1)$ in $t$, the intermediate value theorem guarantees the existence of $\tilde{t}\in(\min\{t_1,t_2\}, \max\{t_1,t_2\})$ such that $F_2(\tilde{t},\delta_1)=0$. 
	\end{itemize}
	To carry out step 1, take $t=3$. A Taylor expansion of $F_2(3,\delta)$ at $\delta=0$ gives
	\begin{align}\label{SSR_example_F2}
		F_2(3,\delta) = \det A_2(3,0)^{\circ\alpha_0} + O(\delta).
	\end{align}
	Now examine the following determinant for $\alpha\in(0,1)$:
	\[\det A_2(3,0)^{\circ\alpha} = -2^{\alpha} + 3^{\alpha} + 4^{\alpha} - 7^{\alpha} - 12^{\alpha} + 14^{\alpha}.\] 
	By Theorem~\ref{Jameson}, this function has at most three real roots. A Taylor expansion at $\alpha=0$ yields $\det A_2(3,0)^{\circ\alpha} = c\alpha^2 + O(\alpha^3)$, where
	\begin{align*}
		c = \frac{1}{2}\left(-(\log2)^2 + (\log3)^2 + (\log4)^2 - (\log7)^2 - (\log12)^2 + (\log14)^2\right).
	\end{align*}
	This shows that $\alpha=0$ is a double root. Furthermore, since $A_2(3,0)$ is a singular SR matrix with sign pattern $\epsilon=(1,1,\epsilon_3)$, where $\epsilon_3\in\{\pm1\}$, Theorem~\ref{Theorem_SR_power_classification} implies that $\det A_2(3,0)^{\circ\alpha}\leq0$ for all $\alpha\in(0,1)$. If equality holds at some $\alpha\in(0,1)$, then $\alpha$ will be a local maximum in $(0,1)$, which results in two more real roots at that point, contradicting Theorem~\ref{Jameson}. Therefore, $\det A_2(3,0)^{\circ\alpha}<0$ throughout $(0,1)$. In particular, $\det A_2(3,0)^{\circ\alpha_0}<0$, so equation~\eqref{SSR_example_F2} implies that for sufficiently small $\delta>0$, say $\delta_1$, we have $F_2(3,\delta_1)<0$.
	
	For step 2, fix $\delta=\delta_1$ and expand $F_2(t,\delta_1)$ at $t=0$ using a Taylor series to obtain
	\begin{align*}
		F_2(t,\delta_1) = \frac{1}{3}\delta_1 \alpha_0^2 t^2 + O(t^3) ~ \implies ~ \displaystyle{\lim_{t\to0^+}\frac{F_2(t,\delta_1)}{t^2}} = \frac{1}{3}\delta_1\alpha_0^2>0.
	\end{align*} 
	Therefore, for sufficiently small $t>0$, say $t_2$, we have $F_2(t_2, \delta_1)>0$. Finally, by the intermediate value theorem, there exists $\tilde{t}\in(t_2,3)$ such that $F_2(\tilde{t},\delta_1)=0$.
	
	To eliminate powers in $(1,\infty)$, consider the matrix
	\begin{align*}
		A_3(t,\delta) := \begin{pmatrix}
			1 & 1 & 1 \\
			1 & 1+2t & 1+3t \\
			1 & 1+6t & 1+(9-\delta)t 
		\end{pmatrix}, \quad 0\leq\delta\leq1, ~ t\in\left[0,\frac{2-\delta}{2\delta}\right).
	\end{align*}
	Observe that all $2\times2$ minors are either of the form $ut$ with $u>0$, or $(v-\delta)t$ with $v>\delta$, or $2t-\delta t-2\delta t^2$; and $\det A_3(t,\delta)= -2\delta t^2$. These conditions ensure that $A_3(t,\delta)$ is an SSR matrix for all $0<\delta\leq1$ and $t\in\left(0,\frac{2-\delta}{2\delta}\right)$. Fix $\alpha_0\in(1,\infty)$, and define the function
	\begin{align*}
		F_3(t,\delta):=\det A_3(t,\delta)^{\circ\alpha_0}, \quad \text{for} ~ 0\leq\delta\leq1, ~ t\in\left[0,\frac{2-\delta}{2\delta}\right).
	\end{align*}
	We claim that there exist $t,\delta>0$ within the specified ranges such that $F_3(t,\delta)=0$. To show this, we follow the same strategy as in the previous case.
	
	First, fix $t_1\in\left(0,\frac{2-\delta}{2\delta}\right)$, and consider the Taylor expansion of $F_3(t_1,\delta)$ at $\delta=0$
	\begin{align}\label{SSR_F3}
		F_3(t_1,\delta) = \det A_3(t_1,0)^{\circ\alpha_0} + O(\delta).
	\end{align}
	To establish the existence of a suitable $\delta_1$, we first show that $\det A_3(t_1,0)^{\circ\alpha}>0$ for all $\alpha>1$. Since $A_3(t_1,0)$ is a singular SR matrix with sign pattern $\epsilon=(1,1,\epsilon_3)$, where $\epsilon_3\in\{\pm1\}$, Theorem~\ref{Theorem_SR_power_classification} ensures that $\det A_3(t_1,0)^{\circ\alpha}\geq0$ for all $\alpha\geq1$. It remains to show that this determinant is positive for all $\alpha>1$. Consider
	\[\det A_3(t_1,0)^{\circ\alpha} = -(1+2t_1)^{\alpha} + (1+3t_1)^{\alpha} + (1+6t_1)^{\alpha} - (1+9t_1)^{\alpha} - (1+9t_1+18t_1^2)^{\alpha} + (1+11t_1+18t_1^2)^{\alpha}.\] 
	By Theorem~\ref{Jameson}, this expression has at most three real roots. Its Taylor expansion at $\alpha=0$ gives
	\begin{align*}
		\det A_3(t_1,0)^{\circ\alpha} = \left(-\log(1+3t_1)\log(1+6t_1) + \log(1+2t_1)\log(1+9t_1)\right)\alpha^2 + O(\alpha^3),
	\end{align*}
	showing a double root at $\alpha=0$. Moreover, since $A_3(t_1,0)$ is singular, $\det A_3(t_1,0)^{\circ\alpha}$ has another root at $\alpha=1$. Thus, $\det A_3(t_1,0)^{\circ\alpha}>0$ for all $\alpha>1$. Now, equation~\eqref{SSR_F3} ensures that for sufficiently small $\delta>0$, say $\delta_1$, we have $F_3(t_1,\delta_1)>0$. 
	
	Now fix this $\delta_1>0$, and expand $F_3(t,\delta_1)$ at $t=0$ to obtain
	\[F_3(t,\delta_1) = -2\delta_1 \alpha_0^2t^2 + O(t^3) ~ \implies ~ \displaystyle{\lim_{t\to0^+}\frac{F_3(t,\delta_1)}{t^2}} = -2\delta_1\alpha_0^2<0.\] 
	Hence, there exists small $t>0$, say $t_2$, such that $F_3(t_2,\delta_1)<0$. By continuity and the intermediate value theorem, there exists $\tilde{t}\in(\min\{t_1,t_2\}, \max\{t_1,t_2\})$ such that $F_3(\tilde{t},\delta_1)=0$.
	
	Finally, for $n\times n$ SSR matrices with $n\geq4$, or for $m\times n$ SSR matrices with $m\neq n$ and $d\geq3$, the matrices above can be extended to SSR matrices of the desired dimensions using Theorem~\ref{Theorem_line_insertion_SSR}. These matrices remain SSR, but their images under the entrywise power function contain a $3\times3$ submatrix with vanishing determinant. This concludes the proof.
\end{proof}

Equipped with the necessary intermediate results, we now present the proof of Theorem~\ref{Theorem_SSR_Non-Fixed_Classification}.

\begin{proof}[Proof of Theorem \ref{Theorem_SSR_Non-Fixed_Classification}]
	That (2)$\implies$(1) is trivial; we show (1)$\implies$(2) to complete the proof. The cases where $m=n=1$, and where $m\neq n$ with $d=1$ follow by Theorem~\ref{Theorem_SSR_all_sign_pos_neg}. It remains to consider the cases $m=n$ with $n\geq3$, and $m\neq n$ with $d=2$. In these situations, Theorem~\ref{Theorem_SSR_all_sign_f_form_power} ensures that $f$ must be of the form
	\begin{align*}
		f(x) = \begin{cases}
			c_1|x|^{\alpha} ~ \text{for some} ~ \alpha,c_1\in\mathbb{R}\setminus\{0\}, & \text{if} ~ x<0, \\
			c_2x^{\alpha} ~ \text{for some} ~ \alpha,c_2\in\mathbb{R}\setminus\{0\}, & \text{if} ~ x>0.
		\end{cases}
	\end{align*}
	The classification of power functions that preserve SSR is given in Theorem~\ref{Theorem_SSR_all_sign_power_classification}. This completes the proof.
\end{proof}

We conclude this section by classifying the functions that entrywise preserve the set of all (strict) sign regular matrices with non-negative (positive) entries. We begin with the case of SR matrices.

\begin{cor}
	Let $m,n\geq1$ be integers, and define $d:=\min\{m,n\}$. Suppose $f:[0,\infty)\to\mathbb{R}$ is a function. Then the following statements are equivalent.
	\begin{itemize}
		\item[(1)] $f[-]$ preserves the set of all $m\times n$ SR matrices with non-negative entries.
		
		\item[(2)] We have the following two cases:
		\begin{itemize}
			\item [(a)] If $m=n$, then
			\begin{itemize}
				\item[(i)] For $n=1,2:$ The function $f$ is non-negative on $[0,\infty)$.
				
				\item[(ii)] For $n=3:$ $f(x)=cx^{\alpha}$ for some $\alpha,c\geq0$.
				
				\item[(iii)] For $n\geq4:$ $f(x)=cx^{\alpha}$ for some $\alpha\in\{0,1\}$ and some $c\geq0$.
			\end{itemize}
			
			\item [(b)] If $m\neq n$, then
			\begin{itemize}
				\item[(i)] For $d=1:$ same form as in case (a)(i).
				
				\item[(ii)] For $d=2:$ same form as in case (a)(ii).
				
				\item[(iii)] For $d\geq3:$ same form as in case (a)(iii).
			\end{itemize}
		\end{itemize}	
	\end{itemize}
\end{cor}
\begin{proof}
	The proof is similar to that of Theorem~\ref{Theorem_SR_Non-Fixed_Classification}, except that the constant ``$c$'' is no longer an arbitrary real number: it must match the sign of the matrix entries as the sign of the $1\times1$ minors of SR matrices must be preserved.
\end{proof}

We state below the analogous result for SSR matrices.
\begin{cor}
	Let $m,n\geq1$ be integers with $(m,n)\neq(2,2)$, and define $d:=\min\{m,n\}$. Let $f:(0,\infty)\to\mathbb{R}$ be a function. Then the following statements are equivalent.
	\begin{itemize}
		\item[(1)] $f[-]$ preserves the set of all $m\times n$ SSR matrices with positive entries.
		
		\item[(2)] We have the following two cases:
		\begin{itemize}
			\item[(a)] If $m=n$, then
			\begin{itemize}
				\item[(i)] For $n=1:$ The function $f$ is positive on $(0,\infty)$.
				
				\item[(ii)] For $n\geq3:$ $f(x)=cx$ for some $c>0$.
			\end{itemize}
			
			\item[(b)] If $m\neq n$, then
			\begin{itemize}
				\item[(i)] For $d=1:$ same form as in case (a)(i).
				
				\item[(ii)] For $d=2:$ $f(x)=cx^{\alpha}$ for some $\alpha\in\mathbb{R}\setminus\{0\}$ and some $c>0$.
				
				\item[(iii)] For $d\geq3:$ same form as in case (a)(ii).
			\end{itemize}
		\end{itemize}
	\end{itemize}
\end{cor}
\begin{proof}
	The proof proceeds in a similar manner to that of Theorem~\ref{Theorem_SSR_Non-Fixed_Classification}, except that only non-zero constants whose sign matches that of the matrix entries are allowed.
\end{proof}

	\subsection*{Acknowledgments}
	We thank Apoorva Khare for a detailed reading of an earlier draft and for providing valuable feedback. We are also grateful to the anonymous referee for carefully going through the manuscript and offering several constructive comments that helped improve the exposition. The first author was partially supported by ANRF Prime Minister Early Career Research Grant ANRF/ECRG/2024/002674/PMS (ANRF, Govt.~of India), INSPIRE Faculty Fellowship Research Grant DST/INSPIRE/04/2021/002620 (DST, Govt.~of India), and IIT Gandhinagar Internal Project: IP/IP/50025.


\begin{thebibliography}{10}
	
	\bibitem{Ando87}
	T.~Ando.
	\newblock Totally positive matrices.
	\newblock \href{https://doi.org/10.1016/0024-3795(87)90313-2}{\em Linear Algebra Appl.}, 90:165--219, 1987.
	
	\bibitem{BGKP16}
	A.~Belton, D.~Guillot, A.~Khare, and M.~Putinar.
	\newblock Matrix positivity preservers in fixed dimension. {I}.
	\href{https://doi.org/10.1016/j.aim.2016.04.016}{\em Adv. Math.}, 298:325--368, 2016.
	
	\bibitem{BGKP20}
	A.~Belton, D.~Guillot, A.~Khare, and M.~Putinar.
	\newblock  Totally positive kernels, P\'olya frequency functions, and their transforms.
	\href{https://doi.org/10.1007/s11854-022-0259-7}{\em J.\ d'Analyse Math.}, 150:83--158, 2023.
	
	\bibitem{BFZ96}
	A.~Berenstein, S.~Fomin, and A.~Zelevinsky.
	\newblock Parametrizations of canonical bases and totally positive matrices.
	\newblock \href{http://dx.doi.org/10.1006/aima.1996.0057}{\em Adv. Math.}, 122:49--149, 1996.
	
	\bibitem{BL08}
	P.J.~Bickel and E.~Levina.
	\newblock Covariance regularization by thresholding.
	\newblock \href{https://doi.org/10.1214/08-AOS600}
	{\em Ann. Statist.}, 36(6):2577--2604, 2008.
	
	\bibitem{BPR12}
	G.~Blekherman, P.A.~Parrilo, and R.R.~Thomas.
	\newblock {\em Semidefinite optimization and convex algebraic geometry}.
	\newblock \href{https://epubs.siam.org/doi/book/10.1137/1.9781611972290}{MOS-SIAM Series on Optimization}, Society for Industrial and Applied Mathematics, 2012.
	
	\bibitem{Bre95}
	F.~Brenti.
	\newblock Combinatorics and total positivity.
	\newblock \href{http://dx.doi.org/10.1016/0097-3165(95)90000-4}
	{\em J.\ Combin.\ Theory Ser.~A}, 71(2):175--218, 1995.
	
	
	\bibitem{BJM81}
	L.D.~Brown, I.M.~Johnstone, and K.B.~MacGibbon.
	\newblock Variation diminishing transformations: a direct approach to total positivity and its statistical applications.
	\newblock \href{https://doi.org/10.2307/2287577 }
	{\em J. Amer. Statist. Assoc.}, 76(376):824--832, 1981.
	
	\bibitem{C22}
	P.N.~Choudhury.
	\newblock Characterizing total positivity: single vector tests via linear complementarity, sign non-reversal, and variation diminution. \newblock \href{https://doi.org/10.1112/blms.12601}{\em Bull.\ London Math.\ Soc.}, 54(2):791--811, 2022.
	
	
	\bibitem{CY-SSR_Construction24}	
	P.N.~Choudhury and S.~Yadav.
	\newblock Constructing strictly sign regular matrices of all sizes and sign patterns. 
	\href{https://londmathsoc.onlinelibrary.wiley.com/doi/10.1112/blms.70080}{\em Bull. London Math. Soc.}, 57(7):2077--2096, 2025.
	
	\bibitem{CY-VDP23}	
	P.N.~Choudhury and S.~Yadav.
	\newblock Sign regular matrices and variation diminution: single-vector tests and characterizations, following Schoenberg, Gantmacher--Krein, and Motzkin.  \href{https://doi.org/10.1090/proc/17026}{\em Proc. Amer. Math. Soc.}, 153:497--511, 2025.
	
	\bibitem{CY-LPP24}	
	P.N.~Choudhury and S.~Yadav.
	\newblock Sign regularity preserving linear operators.
	\newblock\href{https://doi.org/10.1112/blms.70209}{\em Bull. London Math. Soc.}, 58, article \#e70209 (24pp.), 2026.
	
	\bibitem{Darboux75}
	G.~Darboux.
	\newblock Sur la composition des forces en statique.
	\href{}{\em Bull. Sci. Math.}, 9:281--299, 1875.
	
	\bibitem{fallat-john}
	S.M.~Fallat and C.R.~Johnson.
	\newblock {\em Totally non-negative matrices}.
	\newblock\href{https://press.princeton.edu/books/hardcover/9780691121574/totally-non-negative-matrices}{Princeton	Series in Applied Mathematics}, Princeton University Press, Princeton, 2011.
	
	\bibitem{FZ00}
	S.~Fomin and A.~Zelevinsky.
	\newblock Total positivity: tests and parametrizations.
	\newblock \href{http://dx.doi.org/10.1007/BF03024444}%
	{\em Math.\ Intelligencer}, 22(1):23--33, 2000.
	
	\bibitem{FZ02}
	S.~Fomin and A.~Zelevinsky.
	\newblock Cluster algebras. {I}. {F}oundations.
	\newblock \href{http://dx.doi.org/10.1090/S0894-0347-01-00385-X}{\em J.\	Amer.\ Math.\ Soc.}, 15(2):497--529, 2002.	
	
	\bibitem{Frobenius1897}
	G.~Frobenius.
	\newblock  \"{U}ber die Darstellung der endlichen Gruppen durch lineare Substitutionen.
	\newblock {\em Sitzungsber. Preuss. Akad. Wiss. Berlin}, 994--1015, 1897.
	
	\bibitem{gantmacher-krein}
	F.R.~Gantmacher and M.G.~Krein.
	\newblock Sur les matrices compl\`etement nonn\'egatives et oscillatoires.
	\newblock \href{http://www.numdam.org/item?id=CM_1937__4__445_0}
	{\em Compositio Math.}, 4:445--476, 1937.
	


	\bibitem{GK50}
	F.R.~Gantmacher and M.G.~Krein.
	\newblock {\em Oscillation matrices and small vibrations of mechanical systems}. (Russian) Moscow -- Leningrad, 1941.
	
	
	\bibitem{GRS18}
	K.~Gr\"{o}chenig, J.L.~Romero, and J.~St\"{o}ckler.
	\newblock Sampling theorems for shift-invariant spaces, {G}abor frames, and totally positive functions.
	\newblock \href{http://dx.doi.org/10.1007/s00222-017-0760-2}{\em Invent.\ Math.}, 211:1119--1148, 2018.
	
	\bibitem{GKR16}
	D.~Guillot, A.~Khare, and B.~Rajaratnam.
	\newblock Preserving positivity for matrices with sparsity constraints.
	\href{https://doi.org/10.1090/tran6669}{\em Trans. Amer. Math. Soc.}, 368:8929--8953, 2016.
	
	\bibitem{GKR_16}
	D.~Guillot, A.~Khare, and B.~Rajaratnam.
	\newblock Critical exponents of graphs.
	\href{https://doi.org/10.1016/j.jcta.2015.11.003}{\em J. Combin. Theory Ser. A}, 139:30--58, 2016.
	
	\bibitem{GKR17}
	D.~Guillot, A.~Khare, and B.~Rajaratnam.
	\newblock Preserving positivity for rank-constrained matrices.
	\href{https://doi.org/10.1090/tran/6826}{\em Trans. Amer. Math. Soc.}, 369:6105--6145, 2017.
	
	\bibitem{HMPV09}
	J.W.~Helton, S.~McCullough, M.~Putinar, and V.~Vinnikov.
	\newblock Convex matrix inequalities versus linear matrix inequalities.
	\newblock \href{https://doi.org/10.1109/TAC.2009.2017087}
	{\em IEEE Trans. Automat. Control}, 54(5):952--964, 2009.
	
	\bibitem{HR11}
	A.~Hero and B.~Rajaratnam.
	\newblock Large-scale correlation screening.
	\newblock \href{https://doi.org/10.1198/jasa.2011.tm11015}
	{\em J. Amer. Statist. Assoc.}, 106(496):1540--1552, 2011.
	
	
	\bibitem{Jameson}
	G.J.O.~Jameson.
	\newblock Counting zeros of generalised polynomials: Descartes' rule of signs and Laguerre's extensions.
	\newblock \href{https://doi.org/10.1017/S0025557200179628}{\em The Mathematical Gazette}, 90(518):223--234, 2006.
	
	
	\bibitem{Karlin64}
	S.~Karlin.
	\newblock Total positivity, absorption probabilities and applications.
	\newblock \href{https://doi.org/10.2307/1993667}{\em Trans. Amer. Math. Soc.}, 111:33--107, 1964.
	
	\bibitem{K68}
	S.~Karlin.
	\newblock {\em Total positivity. {V}ol. {I}}.
	\newblock Stanford University Press, Stanford, CA, 1968.
	
	
	\bibitem{Karlinsplines}
	S.~Karlin and Z.~Ziegler.
	\newblock Chebyshevian spline functions.
	\newblock \href{http://dx.doi.org/10.1137/0703044}{\em SIAM J.\ Numer.\	Anal.}, 3(3):514--543, 1966.
	
	\bibitem{Khare20}
	A.~Khare.
	\newblock Multiply positive functions, critical exponent phenomena, and the Jain-Karlin-Schoenberg kernel. {\em Prepint,}
	\newblock \href{https://arxiv.org/abs/2008.05121v2}{ 	 	arXiv:2008.05121}, 2021.
	
	\bibitem{KW14}
	Y.~Kodama and L.~Williams.
	\newblock K{P} solitons and total positivity for the {G}rassmannian.
	\newblock \href{http://dx.doi.org/10.1007/s00222-014-0506-3}
	{\em Invent.\ Math.}, 198(3):637--699, 2014.
	
	\bibitem{Lo55}
	C.~Loewner.
	\newblock On totally positive matrices.
	\newblock \href{https://doi.org/10.1007/BF01187945}{\em Math. Z.}, 63:338--340, 1955.
	
	\bibitem{Lu94}
	G.~Lusztig.
	\newblock Total positivity in reductive groups.
	\newblock In {\em Lie theory and geometry}, volume 123 of {\em Progr. Math.}, pages 531--568. Birkh\"{a}user, Boston, MA, 1994.
	
	\bibitem{Mot36}
	T.S.~Motzkin.
	\newblock {\em Beitr\"{a}ge zur Theorie der linearen Ungleichungen}.
	\newblock PhD dissert., Basel, 1933 and Jerusalem, 1936.
	
	\bibitem{Ostrowski29}
	A.M.~Ostrowski.
	\newblock \"Uber die Funktionalgleichung der Exponentialfunktion und verwandte Funktionalgleichung.
	\newblock{\em Jber. Deut. Math. Ver.}, 38:54--62, 1929.
	
	\bibitem{pinkus}
	A.~Pinkus.
	\newblock {\em Totally positive matrices}.
	\newblock \href{http://dx.doi.org/10.1017/CBO9780511691713}{Cambridge Tracts in Mathematics, Vol.~181}, Cambridge University Press, Cambridge, 2010.
	
	\bibitem{PS25}
	G.~P\'olya and G.~Szeg\H{o}.
	\newblock
	{\em Aufgaben und Lehrs\"{a}tze aus der Analysis. Band II: Funcktionentheorie, Nullstellen, Polynome Determinanten, Zahlentheorie.} Springer-Verlag, Berlin, 1925.
	
	\bibitem{Ri03}
	K.C.~Rietsch.
	\newblock Totally positive {T}oeplitz matrices and quantum cohomology of partial flag varieties.
	\newblock \href{http://dx.doi.org/10.1090/S0894-0347-02-00412-5}
	{\em J.\ Amer.\ Math.\ Soc.}, 16(2):363--392, 2003.
	
	\bibitem{RLZ09}
	A.J.~Rothman, E.~Levina, and J.~Zhu.
	\newblock Generalized thresholding of large covariance matrices.
	\newblock \href{https://doi.org/10.1198/jasa.2009.0101}
	{\em J. Amer. Statist. Assoc.}, 104(485):177--186, 2009.
	
	\bibitem{Rudin59}
	W.~Rudin.
	\newblock Positive definite sequences and absolutely monotonic functions.
	\newblock \href{https://projecteuclid.org/journals/duke-mathematical-journal/volume-26/issue-4/Positive-definite-sequences-and-absolutely-monotonic-functions/10.1215/S0012-7094-59-02659-6.short}
	{\em Duke Math. J.}, 26(4):617--622, 1959.
	
	\bibitem{S30}
	I.J.~Schoenberg.
	\newblock \"{U}ber variationsvermindernde lineare {T}ransformationen.
	\newblock \href{http://dx.doi.org/10.1007/BF01194637}{\em Math.\ Z.}, 32:321--328, 1930.
	
	\bibitem{Schoenberg42}
	I.J.~Schoenberg.
	\newblock Positive definite functions on spheres.
	\newblock \href{https://projecteuclid.org/journals/duke-mathematical-journal/volume-9/issue-1/Positive-definite-functions-on-spheres/10.1215/S0012-7094-42-00908-6.short}
	{\em Duke Math. J.}, 9(1):96--108, 1942.
	
	\bibitem{Schoenberg46}
	I.J.~Schoenberg.
	\newblock Contributions to the problem of approximation of equidistant data by analytic functions. Part A. On the problem of smoothing or graduation. A first class of analytic approximation formulae.
	\newblock \href{http://dx.doi.org/10.1090/qam/15914}{\em Quart.\ Appl.\ Math.}, 4(1):45--99, 1946.
	
	\bibitem{S55}
	I.J.~Schoenberg.
	\newblock On the zeros of the generating functions of multiply positive sequences and functions.
	\newblock \href{http://dx.doi.org/10.2307/1970073}{\em Ann.\ of Math.\ (2)}, 62(3):447--471, 1955.
	
	\bibitem{Schur11}
	I.~Schur.
	\newblock Bemerkungen zur Theorie der beschr\"{a}nkten Bilinearformen mit unendlich vielen Ver\"{a}nderlichen.
	\newblock \href{https://www.degruyter.com/document/doi/10.1515/crll.1911.140.1/html}
	{\em J. reine angew. Math.}, 140:1--28, 1911.
	
	
	
	\bibitem{Whitney}
	A.M.~Whitney.
	\newblock A reduction theorem for totally positive matrices.
	\newblock \href{http://dx.doi.org/10.1007/BF02786969}
	{\em J.\ d'Analyse Math.}, 2(1):88--92, 1952.
	
	%
	%
	%
	%
	%
	%
	%
	%
	%
	%
	%
	%
	%
	%
	%
	%
	%
	%
	%
	%
	%
	%
	%
	%
	%
	%
	%
	%
	%
	%
	%
	
\end{thebibliography}
\end{document}